\documentclass[12pt]{article}

\usepackage{amsmath,amsthm,amssymb,thm-restate}
\usepackage{tikz,xspace,nicefrac,enumitem,caption}
\usepackage{enumitem}
\usepackage[textwidth=20mm, textsize=small]{todonotes}
\usetikzlibrary{calc}
\tikzstyle{v}=[thick, draw, circle, inner sep=1.5pt]
\usepackage[margin=1in]{geometry}
\usepackage[mathscr]{euscript}

\usepackage[colorlinks=true,citecolor=black,linkcolor=black,urlcolor=blue]{hyperref}
\newcommand{\arxiv}[1]{\href{http://arxiv.org/abs/#1}{\texttt{arXiv:#1}}}

\allowdisplaybreaks

\makeatletter
\g@addto@macro\normalsize{%
  \setlength\abovedisplayskip{5pt plus 3pt minus 3pt}
  \setlength\belowdisplayskip{5pt plus 3pt minus 3pt}
  \setlength\abovedisplayshortskip{5pt plus 3pt minus 2pt}
  \setlength\belowdisplayshortskip{5pt plus 3pt minus 2pt}
}
\makeatother

\def\dfrac#1#2{\lower0.15ex\hbox{\large$\frac{#1}{#2}$}}

\numberwithin{equation}{section}

\def\({\bigl(}
\def\){\bigr)}

\newtheorem{thm}{Theorem}
\newtheorem{cor}{Corollary}
\newtheorem{lem}{Lemma}

\newtheorem{prop}{Proposition}
\theoremstyle{definition}

\theoremstyle{remark}
\newtheorem{rem}{Remark}

\let\eps=\varepsilon

\newcommand{\sB}{\mathscr{B}}                     
\newcommand{\cC}{\mathcal{C}}                     
\newcommand{\sC}{\mathscr{C}}                     
\newcommand{\fC}{\mathfrak{C}}                    
\newcommand{\cE}{\mathcal{E}}                     
\newcommand{\cF}{\mathcal{F}}                     
\newcommand{\sG}{\mathscr{G}}                     
\newcommand{\cI}{\mathcal{I}}                     
\newcommand{\sI}{\mathscr{I}}                     
\newcommand{\cO}{\mathcal{O}}                     
\newcommand{\sP}{\mathscr{P}}                     
\newcommand{\sQ}{\mathscr{Q}}                     
\newcommand{\cV}{\mathcal{V}}                     
\newcommand{\sW}{\mathscr{W}}                     
\newcommand{\sT}{\mathscr{T}}                     
\newcommand{\sF}{\mathscr{F}}                     
\newcommand{\IN}{\mathbb{N}}                      
\newcommand{\Nb}{\textup N}                       
\newcommand{\dg}{\operatorname{deg}}              
\newcommand{\VS}{\operatorname{V}}                
\newcommand{\ES}{\operatorname{E}}                

\newcommand{\es}{\emptyset}                       

\newcommand{\tw}{\operatorname{tw}}               
\newcommand{\sm}{\setminus}                       
\newcommand{\nis}{\#{\textsf{ind}}}

\let\originalleft\left
\let\originalright\right
\renewcommand{\left}{\mathopen{}\mathclose\bgroup\originalleft}
\renewcommand{\right}{\aftergroup\egroup\originalright}

\newcommand{\ncol}[1][q]{\ensuremath{\textsf{\#col}_{#1}}\xspace}
\newcommand{\col}[1][q]{\ensuremath{\textsf{$\exists$col}_{#1}}\xspace}
\newcommand{\nind}{\ensuremath{\textsf{\#ind}}}
\newcommand{\Gbar}{\ensuremath{\bar{G}}}
\newcommand{\Tbar}{\ensuremath{\bar{T}}}
\newcommand{\chiss}{\ensuremath{\chi_{\mathrm{s}}}}

\newcommand{\cs}{clique separator\xspace}

\newcommand{\CCD}{clique cutset decomposition\xspace}
\newcommand{\CCDs}{clique cutset decompositions\xspace}

\newcommand{\wlg}{without loss of generality\xspace}
\newcommand{\ifff}{if and only if\xspace}
\newcommand{\bs}[1]{\boldsymbol{#1}}
\newcommand{\msf}{\textsf}
\newcommand{\bsf}[1]{\ensuremath{\{\mathsf{#1}\}}}

\newcommand{\pth}[1]{\ensuremath{P_{#1}}}

\newcommand{\thksf}[1]{\textup{\sf thick}(\ensuremath{#1})\xspace}
\newcommand{\fulsf}[1]{\ensuremath{\textup{\sf full}(#1)}\xspace}
\newcommand{\thin}[1]{\textup{\sf thin}(\ensuremath{#1})\xspace}

  \definecolor{darkred}{rgb}{0.6,0.0,0.0}
  \definecolor{lightred}{rgb}{1.0 0.8 0.8}
  \definecolor{lightblue}{rgb}{0.8 0.8 1.0}
  \definecolor{darkgreen}{rgb}{0.0 0.5 0.0}
  \definecolor{lightgreen}{rgb}{0.7 1.0 0.7}

\tikzset{empty/.style={}}

\tikzset{ b/.style = { circle 
                     , draw
                     , thick
                     , inner sep = 0pt
                     , fill = black
                     , minimum size = 3.5pt
                     }
        , sb/.style = { circle 
                     , draw
                     , thick
                     , inner sep = 0pt
                     , fill = black
                     , minimum size = 2pt
                     }
        , w/.style = { circle 
                     , draw
                     , thick
                     , inner sep = 0pt
                     , fill = white
                     , minimum size = 3.7pt
                     }
        , sw/.style = { circle 
                     , draw
                     , thick
                     , inner sep = 0pt
                     , fill = white
                     , minimum size = 2pt
                     }
        , bl/.style = { circle 
                     , draw
                     , thick
                     , inner sep = 0pt
                     , fill = blue
                     , minimum size = 5pt
                     },
        rd/.style = { circle 
                     , draw
                     , thick
                     , inner sep = 0pt
                     , fill = red
                     , minimum size = 5pt
                     }
        , g/.style = { circle 
                     , draw
                     , thick
                     , inner sep = 0pt
                     , fill = lightgray
                     , minimum size = 4.0pt
                     }
        , i/.style = { circle 
                     , thick
                     , inner sep = 0pt
                     , minimum size = 4.0pt
                     }
        , R/.style = { circle 
                     , draw
                     , thick
                     , inner sep = 0pt
                     , fill = lightred
                     , minimum size = 5.5mm
                     }
        , B/.style = { circle 
                     , draw
                     , thick
                     , inner sep = 0pt
                     , fill = lightblue
                     , minimum size = 5.5mm
                     }
        , G/.style = { circle 
                     , draw
                     , thick
                     , inner sep = 0pt
                     , fill = lightgreen
                     , minimum size = 5.5mm
                     }
        , U/.style = { circle 
                     , draw
                     , thick
                     , inner sep = 0pt
                     , minimum size = 5.5mm
                     }
        }

\date{25 October 2024}

\title{Thick Forests}
\author{Martin Dyer\thanks{Work supported by
    EPSRC grant EP/S016562/1 ``Sampling in hereditary classes''.}\\
\small School of Computing\\[-0.5ex]
\small University of Leeds\\[-0.5ex]
\small Leeds LS2~9JT, UK\\[-0.5ex]
\small\texttt{m.e.dyer@leeds.ac.uk}\\
\and Haiko M\"{u}ller\footnotemark[1]\\
\small School of Computing\\[-0.5ex]
\small University of Leeds\\[-0.5ex]
\small Leeds LS2~9JT, UK\\[-0.5ex]
\small\texttt{h.muller@leeds.ac.uk}
}

\begin{document}

\maketitle

\begin{abstract}
We consider classes of graphs, which we call \emph{thick graphs}, that have 
the vertices of a corresponding \emph{thin} graph replaced by cliques and the 
edges replaced by cobipartite graphs. In particular, we consider the case of 
\emph{thick forests}, which we show to be the largest class of perfect thick 
graphs. 

Recognising membership of a class of thick $\sC$-graphs is NP-complete 
unless the class $\sC$ is triangle-free, so we focus on this case.
Even then membership can be NP-complete. However, we show that the class of
thick forests can be recognised in polynomial time. 

We consider two well-studied combinatorial problems on thick graphs, 
independent sets and proper colourings. Since determining the independence or 
chromatic number of a perfect graph is known to be tractable, we examine the 
complexity of \emph{counting} all independent sets and colourings in thick 
forests.
 
Finally, we consider two parametric extensions to larger classes of thick 
graphs: where the parameter is the size of the thin graph, and where the 
parameter is its treewidth.
\end{abstract}

\renewcommand{\figurename}{Fig.}
\section{Introduction}\label{sec:intro}
We introduce classes of graph, which we call \emph{thick graphs}, that can be represented with cliques as vertices and connecting cobipartite graphs as edges. We call the cliques \emph{thick vertices}, or \emph{nodes} and the cobipartite graphs \emph{thick edges}, or \emph{links}. The graph given by shrinking the cliques to single vertices will be called the \emph{thin graph}, which is usually far from  unique. The purpose is to simplify the input graph structure in order to facilitate algorithmic applications.

More formal definitions are given in section~\ref{sec:thickgraphs}. General notation and definitions are given in section~\ref{sec:prelims} below.

In particular, we are interested in whether a given graph has any thin graph lying in some particular graph class. This is the \emph{recognition} problem. However, it seems to be NP-complete for most classes of thin graphs, for example bipartite graphs, as we show in section~\ref{sec:subcol}. We also consider the recognition problem when the input graph must itself lie in some given graph class. We give formal definitions in section~\ref{sec:thickgraphs}.

We show in section~\ref{sec:subcol} that only thick triangle-free graphs can be easily recognised, with a shorter proof than one previously given in~\cite{MacYu}. However, since it is NP-complete to recognise thick bipartite graphs, not all triangle-free thin graphs are recognisable in P. For the triangle-free case,  a fixed-parameter (XP) recognition algorithm was given \cite{MacYu}, when the parameter is the thin graph $H$. But this did not show fixed-parameter tractability (FPT). In section~\ref{sec:fixedparam}, we use the methods developed in section~\ref{sec:recog} to give an FPT algorithm for this class, when the parameter is the number of vertices in the thin graph $H$. We also show that this problem is computationally equivalent to that considered in \cite{MacYu}. 

Since most thick classes lead to NP-complete recognition problems, we focus on a potentially tractable case, where the thin graph must be a \emph{forest}. We show in section~\ref{sec:thickforests} that thick forests are perfect graphs, resembling chordal graphs in some respects. Whereas chordal graphs have clique trees, thick trees are trees of cliques. However, though it is obvious that thick forests are not chordal in general, we  show that not all chordal graphs are thick forests. We develop a polynomial time recognition algorithm for thick forests in section~\ref{sec:recog}. This forms the central section of the paper. 

The \emph{union} of chordal graphs and thick forests is (properly) contained in a larger class $\sQ$ of graphs, which we call \emph{quasi} thick forests, defined in section~\ref{sec:CCD}. We show in section~\ref{sec:CCD} that $\sQ$ is a class of perfect graphs, implying the same result for thick forests, which we give in section~\ref{sec:thickbip}. These can be recognised efficiently by \emph{\CCD}~\cite{Tarjan}, though this does not give efficient recognition for thick forests, as it does for recognition of chordal graphs. Nevertheless, it is more suited to algorithmic applications.

We consider two classical graph problems which are intrinsically related to perfection: independent sets and colourings. That the decision version of these problems is in P for $\sQ$ follows from perfection, though we note that the known algorithms are not combinatorial and require continuous optimisation methods. Therefore we consider the counting versions which are hard for general perfect graphs. Counting colourings and independent sets \emph{exactly} are both \#P-complete for perfect graphs, even for bipartite graphs, and their \emph{approximate} versions are both \#BIS-hard~\cite{DyGoGJ}. So we are interested in classes of graphs for which polynomial time counting is possible.

We show in section~\ref{sec:colouring} that counting colourings exactly remains \#P-complete for thick forests, but approximate counting is in polynomial time for $\sQ$. We use an approach introduced by Dyer, Jerrum, M\"{u}ller and Vu\v{s}kovi\'{c}~\cite{DJMV} for approximately counting independent sets in claw-free perfect graphs via  \CCD. We show here that the technique of~\cite{DJMV} yields polynomial time approximately counting of colourings of graphs in $\sQ$, though this was not possible for claw-free perfect graphs. Note that thick forests are far from claw-free, since they contain all trees. In section~\ref{sec:colouringbeyond}, we show further that even approximate colouring is not possible in any class of thick graphs which is not a subset of thick forests.

By contrast, we show in section~\ref{sec:indsets} that exact counting of (weighted) independent sets can be done in polynomial time for $\sQ$, again using the technique of~\cite{DyGoGJ}, so the counting problem is little harder than for chordal graphs. 

Since forests are precisely the graphs of \emph{treewidth}~1, and thick treewidth-1 graphs are tractable, in section~\ref{sec:treewidth} we consider thick treewidth-$k$ graphs, with $k>1$ a small constant.
We regard this as a parameterisation of the thin graph, where the parameter is treewidth. 
We show that recognition of this class in NP-complete for all $k>1$, so we are obliged to assume that a solution of the recognition problem is supplied. That is we are given a thin graph of treewidth at most $k$.
But even then, we show that counting colourings is hard, even approximately.
However, on the positive side, we give an XP algorithm for exactly counting independent sets. 

\subsection{Previous work}\label{sec:previous}

Thick graphs appear to have been explicitly introduced by Salamon and Jeavons~\cite{SalJ} in the context of the Constraint Satisfaction Problem (CSP). The satisfiability problem is then equivalent to finding the independence number of a thick graph where the thin graph is in some class. In particular, the class of thick trees was considered in~\cite{SalJ}, since the resulting tree-structured CSP can be solved efficiently.

Thick trees have also been studied as \emph{unipolar} graphs, e.g.~\cite{EW,McYo,Tysh}. In our terminology, these are \emph{thick stars}. So $G$ is a unipolar graph if it has nodes comprising a \emph{hub} $\bs h$ and leaves $\bs b_1, \bs b_2, \ldots, \bs b_s$, the \emph{satellites} and only the links $\bs h\bs b_i$ ($i\in[s]$). Polynomial time algorithms for the recognition problem have been given in~\cite{EW},~\cite{McYo} and~\cite{Tysh}. We give another in section~\ref{sec:unipolar} below. Since a star is a tree with at most one interior vertex, these are a subclass of thick forests.

More recently,  Kanj, Komusiewicz, Sorge and van Leeuwen~\cite{KKSL} have generalised this to give an FPT (fixed-parameter tractable) algorithm for recognising thick bipartite graphs using a bound on the size of one of the parts as the parameter. Thus these graphs have a ``hub'' which is a cluster graph of fixed size.

The thick graph concept is implicit in \emph{subcolourings} of graphs, introduced by Albertson, Jamison, Hedetniemi and Locke~\cite{AJHL}. A $q$-subcolourable graph is precisely a thick $q$-colourable graph. The recognition problem is to decide whether $G$ has any $q$-colourable thin graph $H$. In particular,  2-subcolourability is the problem of recognising a thick bipartite graph. Subcolourability has generated a literature of its own. See, for example,~\cite{BFNG,FJLS}.

Stacho~\cite{Stacho} considered 2-subcolourable chordal graphs. We show in section~\ref{sec:forbidden} that these are precisely chordal thick forests. Stacho gave a polynomial time recognition algorithm for this case, but his approach does not extend to general thick forests. We show in section~\ref{sec:recog} that all thick forests can be recognised in P, which implies Stacho's result as a special case. 

The case where the thin graph is a fixed graph $H$ was called an $(H,K)$-partition of the input graph $G$ by MacGillivray and Yu~\cite{MacYu}. As mentioned above, they gave an almost complete solution to the recognition problem in this case. They showed that the recognition problem is NP-complete unless the thin graph is triangle-free.  We review and improve their work in section~\ref{sec:parameter}. 

A similar question to that of \cite{MacYu} is that of \emph{matrix partitions}, e.g.~\cite{FederHell,THJS}. A fixed symmetric matrix $H$ is given which determines the structure of valid partitions. If the matrix is the adjacency matrix of $H$ but with 1's on the diagonal, 0's for the non-edges and $*$'s for the edges, then an $H$-partition of $G$ is precisely an $(H,K)$-partition of $G$. 

\subsection{Preliminaries}\label{sec:prelims}

We summarise our notation and definitions. Most readers will be able to use this section purely for reference.

$\IN$ will denote the set of positive integers. We write $[m] = \{1,2,\ldots, m\}$ for any $m\in \IN$,
For any integers $a\geq b\geq 0$, we write $(a)_b = a(a-1)\cdots (a-b+1)$
for the \emph{falling factorial}, where $(a)_0=1$ for all $a\geq0$, and $(a)_b=0$ if $b>a$ or $b<0$.

Throughout this paper, all graphs are simple finite and undirected.
For a graph $G=(V,E)$ we set $\VS(G)=V$, $\ES(G)=E$, with $n=|V|$ and $m=|E|$.
If $v\in \VS(G)$ or $e\in \ES(G)$ we may simply write $v\in G$ or $e\in G$.
We will write $G\cong H$ if $G$ is isomorphic to $H$.
If $e=\{u,v\}$ for $u,v\in V$, we will write $e=uv$.
A \emph{clique} is a graph having all possible edges, i.e. $E=V^{(2)}=\{vw : v,w\in V, v \ne w\}$.
A \emph{maximal} clique is not a proper subset of any clique in $G$. If $C$ is a clique 
and $A$ is a maximal clique containing it, we will say that we can \emph{expand} $C$ to $A$.

We use $A\uplus B$ to denote the disjoint union of sets $A,B$. The disjoint union
of graphs $G_1,G_2$ is then $G_1\uplus G_2=(V_1\uplus V_2, E_1\uplus E_2)$.
A \emph{cluster graph} is a disjoint union of cliques.

The \emph{complement} $\Gbar=(V,\bar{E})$ has $vw\in\bar{E}$ \ifff\ $v \ne w$ and $vw\notin E$.
$G$ is a clique \ifff $\Gbar$ is an \emph{independent set}.
If $G\in\sB$, the class of bipartite graphs, then $\Gbar$ is cobipartite, two cliques joined by the edges of a bipartite graph.

If $B=(V_1\cup V_2, E)$ is bipartite, its \emph{bipartite complement}  is the graph $\ddot{B}=(V_1\cup V_2, \ddot{E})$, where $\ddot{E}=\{v_1v_2:v_1\in V_1,v_2\in V_2, v_1v_2\notin E\}$.

We say $V_1$ is \emph{complete} to $V_2$ in $G$ if $v_1v_2\in E$ for all $v_i\in V_i$ ($i\in[2]$),
and \emph{anticomplete} to $V_2$ if $v_1v_2\notin E$ for all $v_i\in V_i$ ($i\in[2]$).
These relations are symmetrical between $V_1$ and $V_2$.

The term \emph{induced subgraph} will mean a vertex-induced subgraph,
and the subgraph of $G=(V,E)$ induced by the set $S$ will be denoted by $G[S]$.
A subgraph $G[S]$ is a component of $G$ if it is connected and there is no edge from $S$ to $V\sm S$.
If $G[S]\cong  H$ for some $S\subseteq V$, we say $G$ contains $H$.
A forbidden (induced) subgraph for $G$ is a graph $H$ such that $G$ does not contain $H$.

A class of graphs $\sC$ is \emph{hereditary} if for any $G=(V,E)\in\sC$ and any $S\subseteq V$,
we have $G[S]\in\sC$. A class $\sC$ is \emph{additive} (or \emph{union-closed}) if a graph is in $\sC$ whenever all
its connected components are in $\sC$.

The \emph{neighbourhood} $\{u\in V: uv\in E\}$ of $v\in V$ will be denoted by $\Nb(v)$,
and the \emph{closed} neighbourhood $\Nb(v)\cup\{v\}$ by $\Nb[v]$. The neighbourhood $\Nb(S)$ of a set $S\subseteq V$ is the set $\bigcup_{v\in S}\Nb(v)\sm S$, and $\Nb[S]=\bigcup_{v\in S}\Nb[v]$. We write $\dg(v)$ for the degree $|\Nb(v)|$ of $v$ in $G$.  

Two vertices $u,v\in V$ are called \emph{false twins} if $\Nb(u)=\Nb(v)$ and \emph{true twins} if $\Nb[u]=\Nb[v]$.

If $\es\subset V_1,V_2\subset V$, with $V_1\cap V_2=\es$, we denote the edge set $\{vw\in E:v\in V_1, w\in V_2\}$ by $V_1:V_2$. Then $V_1:V_2$ is a \emph{cut} in $G$ if $V=V_1\uplus V_2$.

A set $S\subseteq V$ is a \emph{cutset} or \emph{separator} of $G$ if, for some $\es\neq V_1,V_2\subseteq V$, we have $V=V_1\cup V_2$, $V_1\cap V_2=S$, and there is no edge $vw\in E$ with $v\in V_1\sm S$ and $w\in V_2\sm S$.
If $G[S]$ is a complete graph (clique), then $S$ is a \emph{\cs}. 

A set $S\subseteq V$ is an \emph{independent set} if we have $vw\notin E$ for all $v,w\in S$. The maximum size of an independent set in $G$ is denoted by $\alpha(G)$, the \emph{independence number} of $G$. We consider the problem of counting all independent sets in $G$, and we will denote this number by $\nind(G)$.
The size of the largest clique in $G$ is denoted $\omega(G)$, so $\omega(G)=\alpha(\Gbar)$.
A \emph{matching} $M\subseteq E$ in $G$ is a set of disjoint edges. That is, $e_1\cap_2=\es$ for all
$e_1,e_2\in M$.

A $q$\emph{-colouring} $\psi$ of the vertices of a graph $G$, with colour set $Q=[q]$, is a function $\psi:V\to Q$. The colour classes are the sets $\psi^{-1}(i)$ for $i\in[q]$. It is a \emph{proper} colouring if $\psi(u)\neq \psi(v)$ for all $uv\in E$, that is, all colour classes are independent sets. Here $\psi(S)=\{\psi(v): v\in S\}$ for $S\subset V$. By convention, the adjective ``proper'' is omitted if there is no ambiguity. We consider the problem of counting (proper) $q$-colourings of (the vertices) of $G$. Denote this number by $\ncol(G)$. The corresponding decision problem $\col(G)$ asks $\ncol(G)>0$? The minimum $q$ for which this is true is the \emph{chromatic number} $\chi(G)$ of $G$. A graph is a clique if $\chi(\Gbar)=1$. It is \emph{bipartite} if $\chi(G)\leq 2$, and its complement $\Gbar$ is \emph{cobipartite}.

A graph $G=(V,E)$ is \emph{perfect} if $\chi(G[S])=\omega(G[S])$ for all $S\subseteq V$. A ($k$-)\emph{hole} $G[S]$ in a graph is
an induced subgraph isomorphic to a cycle $C_k$  of length $k\geq4$, a \emph{long hole} if $k\geq 5$. (A 4-hole is often called a \emph{square}). A \mbox{($k$-)\emph{antihole}} is the complement $\bar{C_k}$ of a hole. Note that a 4-antihole $\bar{C_4}=2K_2$ is disconnected, so antihole usually means the complement of a long hole, that is, a long antihole. The \emph{strong perfect graph theorem} (SPGT)~\cite{ChRoST} shows that $G$ is perfect if and only if it contains no induced hole or antihole with an odd number of vertices. This clearly implies the earlier (weak) \emph{perfect graph theorem} (PGT)~\cite{Lovasz}, that $\Gbar$ is perfect if $G$ is perfect. It is easy to see that a bipartite graph is perfect, and hence a cobipartite graph $G$ is perfect by the PGT. In perfect graphs, a $q$-colouring and a maximum independent set can be found in polynomial time~\cite{GLS}, though only by an indirect optimisation algorithm. Exactly counting independent sets or $q$-colourings is \#P-complete in general~\cite{DyeGre} and approximate counting is \#BIS-hard~\cite{DyGoGJ}, even for the class of bipartite graphs. Thus positive counting results are possible only for restricted classes of perfect graphs.

A graph is \emph{chordal} (or \emph{triangulated}) if it has no hole. It is easy to see that it can have no long antihole: A 5-antihole is also a 5-hole, so forbidden. A hole of size 6 or more contains a $2K_2$, so its complement contains a 4-hole and is forbidden. Thus chordal graphs are perfect. A chordal graph can be recognised by finding a \emph{perfect elimination order} of its vertices. See~\cite{RTL}, which gives a linear time algorithm for constructing the ordering.

A graph is \emph{chordal bipartite} if it is bipartite and has no long holes. Note that these graphs are not chordal in general, despite the name.

\emph{Chain graphs}~\cite{Yannakakis} form the class of graphs which exclude triangles, 4-antiholes and 5-holes. Since a chain graph excludes $2K_2$, it excludes any hole or antihole of size greater than five, or path of more than three edges. Its components are connected bipartite graph with nested neighbourhoods. The complement of a chain graph is a cobipartite chordal graph, called a \emph{cochain} graph.
 
For graph theoretic background not given above, see Diestel~\cite{diestel}, for example.

Finally, we use standard terms from computational complexity without comment. For further information see, for example, Bovet and Crescenzi~\cite{BovCre} for a general introduction, Jerrum~\cite{Jerrum} for counting complexity and Fellows~\cite{Fellows} for parameterised complexity.
\subsubsection{Graph classes}
We will consider several related graph classes, so we give an indication of their relationship and our (nonstandard) notation in Fig.~\ref{fig:classes}. Here an upward line denotes strict inclusion, where the non-obvious inclusions will be justified later. Most of these classes are well studied. The main exceptions, $\thksf\sF$ and $\sQ$, are the topic of this paper.
\begin{figure}[hbtp]
  \centerline{
  \begin{tikzpicture}[xscale=2.0, yscale=1.2]
    \node (cc) at (2,0) {cochain};
    \node (tB) at (2,4) {$\thksf{\sB}$};
    \node (tF) at (2,2) {$\thksf{\sF}$};
    \node (BC) at (3,1) {$\thksf{\sB}\cap\sC_{\triangle}$};
    \node (qf) at (3,3) {$\sQ$};
    \node (ch) at (4,2) {$\sC_{\triangle}$};
    \node (P0) at (5,5) {$\sP_0$};
    \node (fo) at (4,0) {$\sF$};
    \node (wc) at (5,3) {weakly chordal};
    \node (cn) at (6,0) {chain};
    \node (cb) at (6,2) {chordal bipartite};
    \node (bs) at (6,4) {biclique separable};
    \node (T) at (3.5,5) {$\thksf\sT$};
    \node (tCq) at (2,6) {$\thksf{\fC_q}$};
    \node (tG) at (3.5,7) {$\thksf{\sG}=\sG$};
    \node (pe) at (5,6) {$\sP$};
    \draw (cc)--(BC)--(tF)--(tB);
    \draw (tF)--(qf)--(P0)--(pe)--(tG);
    \draw (fo)--(BC)--(ch)--(wc);
    \draw (fo)--(cb)--(wc)--(bs)--(P0);
    \draw (cn)--(cb)  (ch)--(qf); 
    \draw (tB)--(T)--(tG);
    \draw[dashed] (tB)--(tCq);
    \draw (tCq)--(tG);
  \end{tikzpicture}}\vspace{3ex}
    \centerline{
    \begin{tabular}{*{5}{l}}
      $\sP$             & perfect graphs & \hspace{10mm} &
      $\sF$             & forests \\
      $\sP_0$           & (long hole)-free perfect graphs &&
      $\sB$             & bipartite graphs ($\fC_2$)\\
      $\sQ$             & quasi thick forests &&
      $\sC_{\triangle}$ & chordal graphs\\
      $\fC_q$           & $q$-colourable graphs $(q>2)$&&
      $\sT$& triangle-free graphs\\
      $\sG$ & all graphs&&$\sG_k$ & graphs with at most $k$ vertices\\
    \end{tabular}}
  \caption{Graph classes and their inclusions}\label{fig:classes}
\end{figure}

\subsubsection{Treewidth}\label{ss:btw def}

A \emph{tree decomposition} of a graph $G=(V,E)$ is a pair $(T,b)$ where
$T=(I,A)$ is a tree and $b$ maps the nodes of $T$ to subsets of $V$ such that
\begin{enumerate}[topsep=0pt,itemsep=0pt,label=(\arabic*)]
\item for all vertices $v \in V$ exist an $i \in I$ such
  that $v \in b(i)$,\label{tw:ver}
\item for all edges $e \in E$ exist an $i \in I$ such that
  $e \subseteq b(i)$, and\label{tw:edg} 
\item for all vertices $v \in V$ the graph
  $T[\{i \mid v \in b(i)\}]$ is connected.\label{tw:str}
\end{enumerate}
Clearly, the connected graph in the third property 
is a subtree of $T$. For $i \in I$, the set $b(i) \subseteq V$ is called
a \emph{bag} of the tree decomposition.

Every graph $G=(V,E)$ has a tree decomposition with just one bag containing
all vertices. That is, $I=\{i\}$, $A=\es$ and $b(i)=V$. If $G$ is a tree then
we can choose a tree $T$ that is a subdivision of $G$. Assuming $V \cap E=\es$,
this can be formalised to $I = V \cup E$, $A = \{ve \mid v \in e, e \in E\}$,
$b(v)=\{v\}$ for all $v \in V$ and $b(e)=e$ for all $e \in E$. 

The \emph{width} of $(I,b)$ is the maximum size of a bag $b(i)$ minus one,
where $i \in I$. The \emph{treewidth} of $G$, denoted by $\tw(G)$, is the
minimum width of a tree decomposition of $G$.

For a clique $K_n$ on $n$ vertices, the tree decomposition with
a single bag is optimal, so $\tw(K_n)=n-1$. For a tree $G$ with at least one
edge, the decomposition mentioned above is optimal, so $\tw(G)=1$.

For every graph $G=(V,E)$ exists a tree decomposition $(T,b)$ of minimum width
such that $|I| \le |V|$, where $T=(I,A)$ as above. Such a $(T,b)$ can be
obtained from a tree decomposition of width $\tw(G)$ by contracting arcs
$ij \in A$ with $b(i) \subseteq b(j)$. In general, contracting an arc $ij\in A$
(and assigning the bag $b(i) \cup b(j)$ to the resulting vertex) leads to
another tree decomposition of $G$, possibly of larger width.

By the inverse operation of splitting nodes we can ensure, for every arc
$ij \in A$, that $b(i)$ and $b(j)$ differ by at most one vertex.
Formally, a tree decomposition $(T,b)$ is \emph{nice} if the tree $T=(I,A)$
has a root $r \in I$ with $b(r)=\es$ such that every node $i \in I$ is of one
of the following types:
\begin{description}[topsep=0pt,itemsep=0pt]
\item[leaf] For every leaf $l$ of $T$ we have $b(l)=\es$.
\item[introduce] An introduce node $i$ has exactly one child $j$,
  $b(i) \sm b(j)$ is a singleton and $b(j) \subset b(i)$ holds.
\item[forget] An forget node $i$ has exactly one child $j$,
  $b(j) \sm b(i)$ is a singleton and $b(i) \subset b(j)$.
\item[join] A join node $i$ has exactly two children $j$ and $k$ such that
  $b(i)=b(j)$ and $b(i)=b(k)$ hold.
\end{description}
A slightly less restricted version of nice tree decompositions was given
in \cite{Kloks}. Every graph $G=(V,E)$ has a nice tree decomposition
$(T,b)$ of minimum width such that $|I| = \cO(|V|)$.

\section{Thick graphs}\label{sec:thickgraphs}
A \emph{thick vertex}, or \emph{node}, is a clique, and a \emph{thick edge}, or \emph{link}, is a cobipartite graph. Informally, a \emph{thick graph} is a graph with nodes and links. Where we use these terms simultaneously, it will imply that vertex\,$\subseteq$\,node\,$\subseteq$\,thick vertex and/or edge\,$\subseteq$\,link\,$\subseteq$\,thick edge. 

Formally, we define a thick graph as follows. Let $G=(V,E)$ and $H=(\cV,\cE)$ be graphs,
where $n=|V|,\nu=|\cV|$. A surjective function $\psi: V \to \cV$ such that
\begin{enumerate}[topsep=0pt,itemsep=0pt]
\item\label{m:ver} for all $\bs v \in \cV$ the set $\psi^{-1}(\bs v)\subseteq V$ is a clique in
  $G$, and
\item\label{m:edg} for all $uv \in E$ we have $\psi(u)=\psi(v)$ or
  $\psi(u)\psi(v) \in \cE$,
\end{enumerate}
will be called a \emph{model} $(H,\psi)$ of $G$. Then $G$ is the thick graph and $H$ is the thin graph. Intuitively, $G$ is given by identifying each vertex in $\bs u\in\cV$ with a node (clique) $\bs u\subseteq V$ and each edge in $\cE$ with a link (cobipartite graph) $\bs{vw}$ in $G$.

If $\bs u\in \cV$ then $\psi^{-1}(\bs u)$ is a clique in $G$, an $r$-clique $\{u_1,u_2,\ldots,u_r\}$ in $G$, we will identify the node $\bs u$ with this clique, so $G[\bs u ]\cong  K_r$. Similarly, if $\bs{vw}\in \cE$, the link $\bs{vw}$ will be identified with the cobipartite graph $G[\bs u\cup\bs w]$. Thus we will omit the term ``thick'' if it is clear from the context.

We will write $\thksf{H}=\{G:G$ has a model $(H,\psi)\}$ and $\thin{G}=\{H:G$ has a model $(H,\psi)\}$. Clearly $H\in\thksf H$ and $G\in\thin G$, by taking $\psi=id$, the identity map. 

Clearly $\thin G$ corresponds to the set of \emph{clique covers} of $G$. For $G$ an independent set, $|\thin G|=1$. For $G$ an $n$-clique, $|\thin G|=B_n$, the $n$th Bell number. For triangle-free $G$,  $|\thin G|$ is the number of matchings in $G$. Thus the size of $\thin G$ is a \#P-complete quantity in general~\cite{Valiant}. On the other hand, \thksf H is always countably infinite unless $H$ is the empty graph $(\es,\es)$. 

The \emph{clique cover number} of $G$ is $\theta(G)=\chi(\Gbar)=\min\{|\cV|:H\in\thin G\}$. So asking if $\theta(G)\leq k$ is equivalent to $G\in\thksf{\sG_k}$, where $\sG_k$ is the class of graphs with at most $k$ vertices.

If $\sC$ is any graph class, we will write the class $\bigcup_{H\in\sC}\thksf{H}$ as \thksf\sC. This class is hereditary if $\sC$ is hereditary, since cluster graphs are hereditary. Thus \thksf{\cdot} is a heredity-preserving operator on graph classes.

It should be observed that graphs in $\thksf\sC$ may have very different properties from graphs in $\sC$.
For example, bipartite graphs are perfect, but a thick bipartite graph is not necessarily perfect, see Fig.~\ref{fig:cycles}. More obviously, a tree has no holes, but even a thick edge can have $\Theta(n^4)$ 4-holes.

Thus $G\in\thksf\sC$ if there is any $H\in\sC$ such that $G\in\thksf{H}$. For 
given $G$ and fixed $\sC$, the problem of deciding if $G\in\thksf\sC$ is the 
\emph{recognition} problem for $\thksf\sC$. More generally, given two 
hereditary classes $\sC_1,\sC_2$ and $G\in\sC_1$, we might ask whether there 
is any $H\in\sC_2$ such that $H\in\thin G$. This is the recognition problem 
for the class $\thksf{\sC_2}\cap\sC_1$, which we will write as  
$\thksf{\sC_2|\sC_1}$. When $\sC_1$ is the class $\sG$ of all graphs, this is 
simply $\thksf{\sC_2}$. When $\sC_2$ is $\sG$ this is $\sC_1$. The class of 
thick graphs which are in $\sC_1$ is precisely $\sC_1$, because $G\in\thksf 
G$. Equivalently, we ask whether there is any $H\in\thin G$ for $G\in\sC_2$ 
such that $H\in\sC_1$. In general, this problem is NP-complete. Let $\sT$ be 
the class of triangle-free graphs. 
\begin{lem}\label{lem:trianglerecog}
  Let $\sC$ be any class such that $\sC\nsubseteq\sT$, then the recognition problem for $\thksf\sC$ is NP-complete.
\end{lem}
\begin{proof}
 Otherwise some $H=(\cV,\cE)\in\sC$ contains a triangle $T=(abc)$. We construct $G$ has follows. Let $G_0=(V_0,E_0)$ be an instance of recognising $G_0\in\thksf T$. This is NP-complete since it is 3-colourability of $\bar{G_0}$. We may assume $G_0$ contains an independent set $\{a_0,b_0,c_0\}$ of size 3. Otherwise $\chi(\bar{G_0})\leq 2$ so $\bar{G_0}$ is bipartite, and hence $G_0$ is cobipartite, and recognition is in $P$.
 See section~\ref{sec:unipolar}. 
 
 To construct $G=(V,E)$ we first replace $T$ by $G_0$, and for each $v\in \cV\sm\{a,b,c\}$, we set $v\in V$. If $vw\in\cE$, then if $w\notin \{a,b,c\}$ set $vw\in E$, and if $w\in \{a,b,c\}$, set $vw_0\in E$. Then $G_0$ has a model $(T,\psi_0)$ \ifff  $G$ has a model $(H,\psi)$ with $\psi(v)=\psi_0(v)$ for $v\in V_0$ and $\psi(v)=v$ for $v\in V\sm V_0$. Thus deciding $G\in\thksf H$ is equivalent to deciding $G_0\in\thksf T$, so is NP-complete.
\end{proof}
Thus tractable recognition is only possible for classes of triangle-free thin graphs. 

\subsection{Full links and full graphs}\label{sec:loose}
A \emph{full link} is a link which is a clique. Equivalently, the cut between its end nodes is a complete bipartite graph. These can be problematic since they can also be viewed as thick vertices. For a thick triangle, if any edge is full, this may be viewed as a thick edge. Consider the triangle $\msf{abc}$ in Fig.~\ref{fig:fulledge}. The heavy lines represent full links, the other solid lines non-full but nonempty links, and the dashed lines non-full but possibly empty links.  If $\msf{ab}$ is the only full link then the triangle $(\msf{abc})$ is the thick edge $\msf{ab:c}$. If both $\msf{ab}$ and $\msf{bc}$ are thick edges, the triangle can be either of the two thick edges, $\msf{ab:c}$ or $\msf{a:bc}$. If all three $\msf{ab}$, $\msf{bc}$ and $\msf{ac}$ are full then the triangle can be any of three thick edges $\msf{ab:c}$, $\msf{a:bc}$, $\msf{ac:b}$, or the thick vertex $\msf{abc}$. Thus there can be a choice in how the triangle is modelled. We will call these \emph{loose} triangles, and we will not model them as thick triangles. A thick triangle with no full links will be called a \emph{stiff} triangle. The possible ambiguity in loose triangles causes a difficulty in the recognition of thick trees in section~\ref{sec:recog}. We will consider this in section~\ref{sec:adjust}. 

We might also have a 4-cycle $(\msf{abcd})$ as in  Fig.~\ref{fig:fulledge}. This represents a thick edge \ifff either of the pairs of non-adjacent edges $\msf{ab},\msf{cd}$ or $\msf{ac},\msf{bd}$ are both full. If all four edges are full, there is a choice of representation as a thick edge. We will call these loose 4-cycles, and any other a stiff 4-cycle. These 4-cycles cannot have chords, or the 4-cycle is two thick triangles with a common link.

Note that, if $G$ contains a stiff triangle or 4-cycle, we can reject it if we are trying to recognise $G\in\thksf\sF$.  

\begin{figure}[htb]
  \centering
  \begin{tikzpicture}[xscale=0.8,yscale=1,font={\sffamily\scriptsize}]
  \begin{scope}
    \node[v] (a) at (4,6) {a};
    \node[v] (b) at (3,5) {b};
    \node[v] (c) at (5,5) {c};
    \draw[ultra thick] (a)--(b) ;
    \draw[thin] (a)--(c)--(b);
  \end{scope}
  \begin{scope}[xshift=4.5cm]
    \node[v] (a) at (4,6) {a};
    \node[v] (b) at (3,5) {b};
    \node[v] (c) at (5,5) {c};
    \draw[ultra thick] (a)--(b)--(c) ;
    \draw[thin] (a)--(c);
  \end{scope}
  \begin{scope}[xshift=9cm]
    \node[v] (a) at (4,6) {a};
    \node[v] (b) at (3,5) {b};
    \node[v] (c) at (5,5) {c};
    \draw[ultra thick] (a)--(b)--(c)--(a) ;
  \end{scope}
  \begin{scope}[xshift=-0.5cm]
    \node[v] (a) at (6,4) {a};
    \node[v] (b) at (8,4) {b};
    \node[v] (d) at (6,3) {d};
    \node[v] (c) at (8,3) {c};
    \draw[ultra thick] (a)--(b) (c)--(d) ;
    \draw[thin] (b)--(c) (d)--(a) ;
  \end{scope}
  \begin{scope}[xshift=4.5cm]
    \node[v] (a) at (6,4) {a};
    \node[v] (b) at (8,4) {b};
    \node[v] (d) at (6,3) {d};
    \node[v] (c) at (8,3) {c};
    \draw[ultra thick] (a)--(b)--(c)--(d)--(a) ;
  \end{scope}
  \end{tikzpicture}\caption{Loose triangles and 4-cycles}\label{fig:fulledge}
\end{figure}

We will call a thick graph which has only full links a \emph{full graph}. These form a subclass of thick graphs. They can clearly have no stiff triangles. So it might seem from the above that we can assume that full graphs are triangle-free, but this is not so. Reducing a full triangle to a thick edge or vertex will generally lead to a graph which is no longer full.

To illustrate, consider the graph in Fig.~\ref{fig:fullgraph}. The only model as a full graph is the identity. It has two models as a stiff triangle with nodes $\{0,1\},\{2,3\},\{4,5\}$ and $\{1,2\},\{3,4\},\{5,0\}$ respectively. It has six models as a thick \pth2, with the nodes
$\{i,i+1,i+2\},\{i+3,i+4,i+5\}$ $\pmod 6$ for $i\in[6]$.

\begin{figure}[hbtp]
\centerline{\begin{tikzpicture}[scale=1.3,line width=0.75pt,font={\sffamily\scriptsize},rotate=90]
\foreach \x in {0,1,2,3,4,5}{\coordinate (\x) at (\x*60:1) {};}
    \foreach \x in {0,1,2,3,4,5}{\node[v] (v\x) at (\x) {\x};}
    \draw (v0)--(v1)--(v2)--(v3)--(v4)--(v5)--(v0);
    \draw (v0)--(v2)--(v4)--(v0) (v1)--(v3)--(v5)--(v1);
\end{tikzpicture}}\caption{Full links}\label{fig:fullgraph}
\end{figure}

Let \fulsf H denote the class of graphs obtained by replacing the vertices of the graph $H$ with cliques and its edges by full links. Then, if $\sC$ is any graph class, let $\fulsf\sC=\bigcup_{H\in\sC}\fulsf{H}$. Thus \fulsf\cdot is again an operator on graph classes. Clearly $\fulsf\sC\subset\thksf\sC$ for any $\sC$.

In general, graphs in $\thksf\sC$ usually have very different properties from graphs in $\sC$.
For example, bipartite graphs are perfect, but a thick bipartite graph is not necessarily perfect. See Fig.~\ref{fig:cycles}. More obviously, a tree has no holes, but a link may have an exponential number of 4-holes. By contrast, $\fulsf\sC$ inherits many properties from $\sC$.

\begin{figure}[hbtp]
\centerline{\begin{tikzpicture}[line width=0.5pt]
\begin{scope}
    \foreach \x in {0,1,2,3}{\coordinate (\x) at (\x*90:1) {};}
    \foreach \x in {0,1,2,3}{\node[v,fill=black] (v\x) at (\x) {};}
    \draw (v1)--(v2)--(v3)--(v0)--(v1);
\end{scope}
\begin{scope}[xshift=4.5cm]
    \foreach \x in {0,1,2,3,4}{\coordinate (\x) at (54+\x*72:1) {};}
    \foreach \x in {0,1}{\node[v,fill=black] (v\x) at (\x) {};}
    \foreach \x in {2,3,4}{\node[v,fill=black] (v\x) at (\x) {};}
    \draw (v1)--(v2)--(v3)--(v4)--(v0);
    \draw[line width=2.5pt] (v0)--(v1);
\end{scope}
\begin{scope}[xshift=9cm]
    \foreach \x in {0,1,2,3,4}{\coordinate (\x) at (54+\x*72:1) {};}
    \foreach \x in {0,1}{\node[v,fill=black] (v\x) at (\x) {};}
    \foreach \x in {2,3,4}{\node[v,fill=black] (v\x) at (\x) {};}
    \draw (v1)--(v2)--(v3)--(v4)--(v0)--(v2) (v1)--(v4);
    \draw[line width=2.5pt] (v0)--(v1);
\end{scope}
    \end{tikzpicture}}
\caption{$C_4$, a thick $C_4$ and a full $C_4$}
  \label{fig:cycles}
\end{figure}

The following is essentially the content of section 4.2.1 of \cite{DGM}. 
\begin{lem}[Dyer,\,Greenhill,\,M\"uller]\label{lem:full}
Let $\cF$ be a set of graphs, none of which contains true twins, and let $\sC$ be the graph class with minimal forbidden subgraphs $\cF$. Then $\fulsf\sC=\sC$.
\end{lem}
\begin{proof}
  Clearly $H\in\fulsf H$, because a vertex is a 1-clique and an edge is a 2-clique. Thus $\sC\subseteq\fulsf\sC$. That $\fulsf H\subseteq\sC$ if $H\in\sC$ is Lemma~5 of~\cite{DGM}.
\end{proof}
Thus, for classes $\sC$ satisfying the assumption of Lemma~\ref{lem:full},  \fulsf\cdot is the identity operator. Many hereditary classes satisfy Lemma~\ref{lem:full}: unit interval, chordal, weakly chordal, perfect and even-hole-free graphs, for example. 

It is also proved in Lemma~5 of \cite{DGM} that, for any class having a forbidden subgraph containing true twins,
Lemma~\ref{lem:full} no longer holds. Thus it is false, for example, in any subclass of triangle-free graphs, 
since any two vertices of a $k$-clique are true twins for all $k>2$.
 
Lemma~\ref{lem:full} can be useful. It was used in~\cite{DGM} to simulate vertex weights in the independent set problem for graphs of bounded bipartite pathwidth. It was used by Lovasz~\cite{Lovasz} in the proof of the (weak) PGT, although without explicitly assuming that the class $\sP$ satisfies the condition of Lemma~\ref{lem:full}, which follows from the SPGT. 

The recognition problem for \thksf\sC is the same as that for $\sC$ for any class satisfying Lemma~\ref{lem:full}, but
that is not true for subclasses. For example, $\sB\subset\sP$, and $\sP$ satisfies Lemma~\ref{lem:full}, so $\thksf\sP=\sP$. But, even though the recognition problem for $\sB$ is almost trivial, and recognition algorithms for $\sP$ are known~\cite{CSSS}, the status of the recognition problem for \fulsf\sB is unclear. Even recognising \fulsf\sF is a different question from recognising \thksf\sF, even though $\fulsf\sF\subset\thksf\sF$. For example, a $k$-star has $k+1$ models in $\thksf\sF$, but a unique model in \fulsf\sF. On the other hand, graphs in $\fulsf\sF$ are chordal, whereas those in $\thksf\sF$ are generally not.

So, full graph classes are not without interest, but we do not consider them further here.

\subsection{Subcolouring}\label{sec:subcol}
The \emph{subchromatic number}~\cite{AJHL} $\chiss(G)$ of $G$ is $\min\{\chi(H):H\in\thin G\}$ is. A graph is \emph{$q$-subcolourable} if $\chiss(G)\leq q$. This is equivalent to asking if $G\in\thksf{\fC_q}$, where $\fC_q$ is the class of $q$-colourable ($q$-partite) graphs. 

A graph $G$ has $\chiss(G)=1$ \ifff it is a \emph{cluster graph}, or $G\in\thksf\sI$, where $\sI$ is the class of independent sets. Also $\chiss(G)=2$ \ifff $G\in\sB$, where $\sB=\fC_2$ is the class of bipartite graphs. We will be interested in the class $\thksf\sF$ of \emph{thick forests}, where $\sF\subset\sB$ is the class of \emph{forests}. We consider the recognition problem for thick forests in section~\ref{sec:recog}, and counting independent sets and colourings in sections~\ref{sec:indsets} and~\ref{sec:colouring}.

A \emph{generalised $q$-colouring} of a graph $G=(V,E)$ with a graph class $\cC$ is a partition of $V$ into subsets $V_1,V_2,\ldots,V_q$ such that $G[V_i]\in\cC$ for all $i\in[q]$. If $\sC$ is the class of cluster graphs, then a generalised $q$-colouring is a $q$-\emph{subcolouring}~\cite{AJHL}. So $G$ has  1-subcolouring if it is a cluster graph and has a 2-subcolouring if it is a thick bipartite graph. Unfortunately, whether $G$ has a $q$-subcolouring is an NP-complete question for any $q\geq2$~\cite{farrugia}. 

\begin{figure}[hbtp]
\centerline{\begin{tikzpicture}[line width=0.5pt]
    \foreach \n/\x in {red/90, blue/210,darkgreen/330}{
    \foreach \y in {1,2,3,4}{
    \node[v,fill=\n] (\n\y) at (\x:0.5+0.5*\y) {};}
    \draw[\n] (\n1)--(\n2)--(\n3)--(\n4) ;
    \draw[\n] (\n4) to[out=\x-125,in=\x-50] (\n1);
    \draw[\n] (\n4) to[out=\x-140,in=\x-50] (\n2);
    \draw[\n] (\n1) to[out=\x+50,in=\x+120] (\n3);}
    \foreach \xx in {2,4}{\foreach \yy in {1,4}
    \draw[line width=0.5pt] (red\xx)--(blue\yy) (blue\xx)--(darkgreen\yy) (darkgreen\xx)--(red\yy);}
  \end{tikzpicture}}
\caption{A thick triangle with a 3-subcolouring}
  \label{fig:thicktriangle}
\end{figure}

\begin{lem}\label{lem:thickNP}
It is NP-complete to decide  whether $G\in\thksf{\fC_q}$ for any $2\leq q=\cO(n^{1-\eps})$, where $0<\eps<1$ is any constant.
\end{lem}
\begin{proof}
  This follows from a result of Farrugia~\cite{farrugia} (see also~\cite{Brown}) on generalised colourings. Let $\mathscr{C}_1,\mathscr{C}_2$ be two additive hereditary graph classes with recognition in NP, neither being the class of edgeless graphs. Let  $\psi$ be a 2-colouring of an input graph $G=(V,E)$ with $C_i=\psi^{-1}(i)$ so that $G[C_i]\in\mathscr{C}_i$ for $i\in[2]$. Then it is NP-complete to decide whether such a 2-colouring exists. For a bipartite thick graph, we take $\mathscr{C}_1$ and $\mathscr{C}_2$ to be the cluster graphs. 
This class is clearly additive and hereditary, and recognition is in P. Then the existence of $\psi$ is equivalent to asking whether $G$ is a thick bipartite graph.

This can then be extended to asking whether $G$ is a thick $q$-partite graph by inductively taking the same $\mathscr{C}_1$ and taking $\mathscr{C}_2$ to be $\thksf{\fC_{q-1}}$. This remains valid if $q=\cO(n^{1-\eps})$,  since polynomial in $n^{\eps}$ is equivalent to polynomial in $n$.
\end{proof}

Note that the additive condition used in Lemma~\ref{lem:thickNP} is important, or recognition of unipolar graphs would be NP-complete when it is not, as we discuss in section~\ref{sec:unipolar}.  Also, Lemma~\ref{lem:thickNP} does not necessarily hold if $G$ is restricted to lie in some graph class. For example, deciding whether $G\in\thksf{\fC_q|\sC}$ is shown in~\cite{BFNG} to be in P if $\sC$ is the class of interval graphs or permutation graphs. We note that these three classes are subclasses of perfect graphs $\sP$. However perfection is clearly not a sufficient condition, since $\sB\subset\sP$. Moreover, $q$-subcolouring for $q>2$ is NP-complete even in the class of chordal graphs~\cite{Stacho}. That is, the recognition problem for $\thksf{\fC_q\mid\sC_\triangle}$ is NP-complete for $q>2$.

\subsection{Altering the model}\label{sec:mod}
Given a model $(H,\psi)$ of $G$ we will consider how $\psi$ can be changed so that $H'$ corresponds to an operation on $H$. In particular, we are interested in modifications which give $H'=H$.

Given $G$ and $H,H'\in\thin G$ we say that $u\in\bs u\in \cV$ is \emph{movable} to $\bs w\in \cV'$ if the model $(\psi,H)$ of $G$ can be modified to  a model $(\psi',H')$ of $G$ so that $\psi'(u)=\bs w'=\bs w\cup\bs u$ and $\psi'(v)=\psi(v)$ for all $v\neq u$. 
\begin{lem}\label{lem:movable}
 A vertex $u\in\bs u$ is movable to $\bs w$ \ifff $\bs u\cup \bs w\subseteq\Nb[u]$.
\end{lem}
\begin{proof}
  We must have $\bs w\cup u$ a clique, which is true \ifff $u$ is complete to $\bs w$, which is equivalent to $\bs w\subseteq\Nb(u)$. Also $\bs u\sm u\subseteq\Nb(u)$ since $\bs u$ is a clique. Conversely, if $\bs w\cup \bs u\nsubseteq\Nb[u]$ then there is a $w\in\bs w$ such that $uw\notin \ES(G)$. So $\bs w\cup u$ is not a clique,
  a contradiction. 
\end{proof}
\begin{cor}\label{cor:neighbour}
A vertex $u\in \bs u\in \cV$ is movable to $\bs w\in \cV'$ only if 
$\bs{uw}\in \cE$ or $\bs w=\{u\}$.\qed 
\end{cor}
It follows that $\thin G$ is connected by vertex moves, since any $H\in\thin G$ is connected to $G\in\thin G$ by a sequence of $\cO(n)$ moves, by dismantling the nodes of $H$ to single vertices, then reassembling them to give $H'$. 

Most moves in $G$ change $H$, since moving $u$ can produce a new edge in  $H'$. We will call move the \emph{good} if $H'=H$ after the move. We could give a general criterion for move to be good, but we will restrict attention to the triangle-free case, $H\in\sT$.

\begin{lem}\label{lem:goodmove}
If $H\in\thin G\cap\sT$ and $\bs{u,w}\in H$ then moving $u\in\bs u$ to $\bs w$ is a good move \ifff $\{u\}\subset\bs u$ and $\Nb[u]=\bs u\cup\bs w$. The reverse of a good move is a good move.
\end{lem}
\begin{proof}
  If $\{u\}\subset\bs u$ and $\Nb[u]=\bs u\cup\bs w$ then moving $u$ to $\bs w$ does not change $H$, since $u$ is movable by Lemma~\ref{lem:movable} and has no neighbours outside $\bs u$ and $\bs w$. Conversely, suppose $v\in\Nb(u)$ but $v\notin\bs u\cup\bs w$. Then $u\in\bs w'\neq\bs w$ for some neighbour $\bs w'$ of $\bs u$. Thus moving $u$  produces an edge $uv$ between $\bs w\cup u$ and $\bs w'$, so the resulting $H'\in\thin G$ contains a triangle $\bs u\sm u$, $\bs w\cup u$, $\bs w'$, so $H'\notin\sT$ and hence $H'\neq H$.  If $\{u\}=\bs u$, then moving $u$ to $\bs w$ contracts the edge $\bs{uw}$ in $H$ to the vertex $\bs w$ in $H'$ and so $H'\neq H$. Good moves are reversible from the symmetry of $\Nb[u]=\bs u\cup\bs w$.
\end{proof}
In fact, we will also call a move good if $\bs w=\es$. In this case we have $\Nb[u]=\bs u$, and the move corresponds to $H$ having a new leaf vertex $\bs w'=\{u\}$. The reverse move then contracts the full link $\bs w'\bs u$ to $\bs u$. 

Successive moves are even more restricted if we require $H'=H\in\sT$.
\begin{lem}\label{lem:neighbour}
If $H\in\thin G\cap\sT$ and $u_1,u_2\in\bs u$ have good moves to $\bs w_1,\bs w_2$ respectively. Then both can be moved so that $H'=H\in\sT$ if and only if $\bs w_1=\bs w_2$ and $\{u_1,u_2\}\subset\bs u$. 
\end{lem}
\begin{proof}
  Both $\bs w_1,\bs w_2$ are neighbours of $\bs u$. If $\bs w_1\neq\bs w_2$, moving gives $\bs u'=\bs u\sm u_1$ and  $\bs w'_1=\bs w_1\cup u_1$. But moving $u_2$ to $\bs w_2$ is a good move only if $\bs h'\cup \bs w_2=\Nb(u_2)$,
  which is false since $u_1\in\Nb(u_2)$, but $u_1\in\bs h'\cup \bs w_2$ only if $\bs w_1=\bs w_2$.
  
  If $\bs w_1=\bs w_2$, both  can be moved in either order unless $\{u_1,u_2\}=\bs u$, in which case 
  moving both replaces the path $\bs w_1\bs u\bs w_2$ by an edge $\bs{w_1w_2}\notin \cE$, since $H\in\sT$.
  Thus $H'\neq H$. 
\end{proof}
\begin{cor}\label{cor:oneneighbour}
If $H\in\thin G\cap\sT$, a subset $U\subseteq\bs u$ can have good moves to at most one of $\bs u$'s neighbours $\bs w$ so that $H'=H$. \qed 
\end{cor}

While all models of $G$ are connected by moves, they need not be connected by 
good moves. We will call a model \emph{rigid} if it has no good moves. 
Consider $G$ the 4-cycle (\msf{abcd}) with $H$ a single edge $\bs u\bs w$. 
There are two models $\bs u=\{a,b\},\bs w=\{c,d\}$ or $\bs u=\{a,d\},\bs 
w=\{b,c\}$, but both are rigid. Thus neither is connected to the model 
$(G,id)$ by good moves.

\emph{Contracting} a full link $\bs{uw}$ in $H$ corresponds to regarding it as a 
node $\bs{u\cup w}$ in $H'$, that is, contracting the edge of $H$. 
This may also be viewed as moving all vertices from $\bs u$ into $\bs 
w$, or vice versa. Note that these are not necessarily good moves. We used 
this operation in section~\ref{sec:loose} to eliminate loose triangles and 
4-cycles. But, if $H\in\sT$, and we require $H'\in\sT$, we cannot contract an 
edge of a stiff 4-cycle, or we obtain a stiff triangle. 

We will call the reverse operation \emph{splitting} the thick vertex $\bs u$. 
We divide $\bs u$ into disjoint subsets $\bs u_1,\bs u_2$ so that 
$\bs{u_1u_2}$ becomes a full link in $H'$. If $H\in\sT$ and we require 
$H'\in\sT$, we must ensure that $\Nb(\bs u_1)$ and $\Nb(\bs u_2)$ are 
disjoint. Also, we cannot usefully split $\bs u_1$ or $\bs u_2$, or we 
introduce a loose triangle. 

We have seen that cannot hope to recognise the whole class \thksf\sT, since $\sB\subset\sT$ and we know from Lemma~\ref{lem:thickNP} that recognition of \thksf\sB is NP-complete. So our approach to recognition of graphs in $\thksf\sT$ will be based on another property of triangle-free graphs, that the induced neighbourhood every vertex is a star. This corresponds to a \emph{unipolar} subgraph of $G$. Recall the definition from section~\ref{sec:previous}. We use this to try to construct $H$ one vertex at a time and then, if necessary, move vertices as described above to ensure consistency.

So we require recognition of unipolar graphs, which is in polynomial time. See~\cite{EW,McYo,Tysh} and section~\ref{sec:unipolar} below. But the following underlies our approach. 
\begin{lem}\label{lem:modelunipol}
  The model $(H,\psi)$ of a unipolar graph $G$ is 
  \begin{enumerate}[itemsep=0pt,topsep=0pt,label=(\alph*)]
    \item unique up to good moves between the hub and one of its satellites;\label{normalcase}
    \item rigid, but with a possible choice between two models.\label{othercase}
  \end{enumerate}
\end{lem}
\begin{proof}
  We assume a hub $\bs h$ with satellites $\bs b_1,\bs b_2,\ldots\bs b_s$, and model $(H,\psi)$.
  If $\bs h$ is given, then $G\sm\bs h$ is a cluster graph, so $\bs b_1,\bs b_2,\ldots\bs b_s$ are uniquely identified. If $u\in\bs b_i$, we cannot move to $\bs b_j$ for $i\neq j$ by Lemma~\ref{lem:movable}, since $\bs b_j\cap\Nb(u)=\es$. We can move $u\in\bs b_i$ to $\bs h$ by Lemma~\ref{lem:movable} if $\bs b_i\cup\bs h\subseteq\Nb(u)$.  But clearly $\Nb(u)\subseteq\bs b_i\cup\bs h$, so we must have $\bs b_i\cup\bs h=Nb(u)$. Thus this is a good move by Lemma~\ref{lem:goodmove}. We cannot move vertices into more than one satellite by Cor.~\ref{cor:oneneighbour}. This establishes the ``normal'' case~\ref{normalcase},  which is sufficient if $G$ is chordal. See~\cite{EW}.
  
However, there is another case, where $\bs h$ has a subclique $\msf a$ such 
that $G\sm\msf a$ has a thick $P_3$ component \{\msf{b,c,d}\}. See 
Fig.~\ref{fig:othercase}. Then $(\msf{abcd})$ is a thick 4-cycle. This is a 
thick edge if either of the pairs of links \msf{ab,cd} or \msf{ac,bd}  
are both full. If as shown in Fig.~\ref{fig:othercase}, $\msf{ac,bd}$ are 
full links, then \msf{ac} is the hub, \msf{bd} is a satellite and $G$ is 
unipolar. If all of \msf{ab,cd,ac,bd} are full, and $(\msf{abcd})$ is the only
thick 4-cycle, then there is a choice of whether \msf{ab} or \msf{ac} becomes the hub. 
  
  However, there can be more than one such 4-cycle. Each must contribute a full link to the hub so as to be contained in a thick edge. Thus, if there is more than one, the hub must be their unique common full link for $G$ to be unipolar. An example is shown in Fig.~\ref{fig:othercase}, where  \msf{ac} is the hub and \msf{bd}, \msf{ef} are the satellites.
  
  Note also that the model is rigid. There are no good moves from any satellite because no vertex in any satellite is complete to the hub (\msf{ac} in Fig.~\ref{fig:othercase}). Since good moves are reversible, neither can  there be any from the hub into the satellites.
\end{proof}
  
\begin{figure}[htb]
  \centering
  \begin{tikzpicture}[xscale=0.8,yscale=1,font={\sffamily\scriptsize}]
    \node[v] (e) at (2,3) {e};
    \node[v] (f) at (3.5,3) {f};
    \node[v] (a) at (4,5) {a};
    \node[v] (b) at (6,4.5) {b};
    \node[v] (c) at (5,3.5) {c};
    \node[v] (d) at (7,3) {d};
    \draw[thin] (b)--(a)--(f) (c)--(d) ;
    \draw[ ultra thick] (e)--(a)--(c) (b)--(d);
  \end{tikzpicture}
  \hspace*{2cm}
  \begin{tikzpicture}[xscale=0.8,yscale=1,font={\sffamily\scriptsize}]
    \node[v] (e) at (3,4) {e};
    \node[v] (f) at (4,3) {f};
    \node[v] (a) at (4,5) {a};
    \node[v] (b) at (6,4.5) {b};
    \node[v] (c) at (5,3.5) {c};
    \node[v] (d) at (7,3) {d};
    \draw[thin] (b)--(a)--(e) (f)--(c)--(d) ;
    \draw[ultra thick](a)--(c) (b)--(d) (e)--(f);
  \end{tikzpicture}\caption{Case~\ref{othercase} of Lemma~\ref{lem:modelunipol}}\label{fig:othercase}
\end{figure}   

\section{Thick forests}\label{sec:thickforests}
A thick tree $G$ with 26 vertices and 49 edges is shown in Fig.~\ref{fig:chordalthicktree}. If we take the nodes to be the horizontal $K_2$'s, the thin tree has 4 internal vertices, 9 leaves and 12 edges.
In fact, it is easy to see that $G$ is a chordal graph, so it is a chordal thick tree.

\begin{figure}[hbt]
{\tikzstyle{every node}=[circle, draw, thick, minimum size=2mm, inner sep=0pt]
\renewcommand{\strut}{\raisebox{-0.5ex}{\rule{0pt}{2.0ex}}}
\newcounter{ctr}
\newcommand{\abc}[1]{\setcounter{ctr}{#1}\alph{ctr}}
\begin{center}
  \begin{tikzpicture}[xscale=0.18, yscale=1.5, font=\sffamily]
    \foreach \n in {1,2,...,18} {
      \ifodd\n \pgfmathtruncatemacro{\x}{4*\n-3}
      \else    \pgfmathtruncatemacro{\x}{4*\n-4}
      \fi
      \node[label=below:{\strut\abc{\n}\strut}] (\n) at (\x,0) {};
    }
    \foreach \n in {19,20,...,24} {
      \ifodd\n \pgfmathtruncatemacro{\x}{12*\n-219}
               \node[label=left:{\strut\abc{\n}\strut}] (\n) at (\x,1) {};
      \else    \pgfmathtruncatemacro{\x}{12*\n-228}
               \node[label=right:{\strut\abc{\n}\strut}] (\n) at (\x,1) {};
      \fi

    }
    \node[label=above:{\strut{y}\strut}] (25) at (33,2.5) {};
    \node[label=above:{\strut{z}\strut}] (26) at (36,2.5) {};
    \foreach \u/\v in { 1/2,   3/4,   5/6}  \draw (\v)--(\u)--(19)--(\v)--(20);
    \foreach \u/\v in { 7/8,   9/10, 11/12} \draw (\v)--(\u)--(21)--(\v)--(22);
    \foreach \u/\v in {13/14, 15/16, 17/18} \draw (\v)--(\u)--(23)--(\v)--(24);
    \foreach \u/\v in {19/20, 21/22, 23/24} \draw (\v)--(\u)--(25)--(\v)--(26);
    \draw (25)--(26);
  \end{tikzpicture}
\end{center}}
\caption{A chordal thick tree}
  \label{fig:chordalthicktree}
\end{figure}
We have shown in section~\ref{sec:subcol} that even recognising a thick bipartite graph is hard in general. By contrast, we will show that the class of thick forests $\thksf\sF\subset\thksf\sB$ is tractable for most purposes. However, some of its properties of \thksf\sF are inherited from \thksf\sB, so we first consider this class.
 
\subsection{Thick bipartite graphs}\label{sec:thickbip}
We have the following characterisation. 
\begin{lem}\label{lem:thickB}
 $G\in\thksf\sB$ if and only if it has a vertex 2-colouring such that each colour class has no induced $P_3$. 
\end{lem}
\begin{proof}
  $P_3$ is the only forbidden subgraph for a cluster graph, and $G\in\thksf\sB$ if and only if its thin graph is 2-colourable.
\end{proof}
From this we have the following.
\begin{lem}\label{oddantiholes}
  A thick bipartite graph has no odd antihole of size greater than five.
\end{lem}
\begin{proof}
  Consider a 2-colouring with colours red and blue of an antihole $\bar{C}_k$ where $k\geq 7$ and odd. Suppose the vertices of $C_k$ are $v_0,v_1,\ldots, v_{k-1}$ cyclically, with subscripts $\bmod~k$. Since $k$ is odd, there must be two adjacent vertices in $C_k$, $v_1,v_2$ say, which have the same colour, say red. These are nonadjacent in $\bar{C}_k$, but have all of $v_4,v_5,\ldots,v_{k-1}$ as common neighbours. These $(k-4)$ vertices must therefore all be coloured blue, or we would have a red $P_3$. Now $v_0$ and $v_3$ cannot both be red, or there would be a red $P_3$: $v_0,v_3,v_1$. So there can be at most three red vertices and, by symmetry, at most three blue vertices. This is a contradiction, since~$k\geq7$.
\end{proof}
From this we have
\begin{lem}\label{lem:bipperfect}
  A thick bipartite graph $G$ is perfect \ifff it has no odd holes.
\end{lem}
\begin{proof}
From the SPGT, we need only show that $G$ contains no odd antihole. But $\bar{C}_5=C_5$ excludes 5-antiholes, and Lemma~\ref{oddantiholes} excludes larger odd antiholes.
\end{proof}
Let $\sP_0$ denote the class of graphs without long holes or odd antiholes, as defined in~\cite{SalJ}. Clearly $\sP_0\subset \sP$, from the SPGT. So  $\sP_0$ is the class of perfect graphs without long holes. Thus we have
\begin{cor}[Salamon and Jeavons~\cite{SalJ}]\label{cor:forestperfect}
  $\thksf\sF\subseteq\sP_0$.
\end{cor}
\begin{proof}
If $G\in\sF$, a chordless cycle must be contained in a thick edge. Thus $G$ can only have triangles and 4-holes.
The result then follows from Lemma~\ref{lem:bipperfect}.
\end{proof}
For thick graph classes, this has a converse.
\begin{lem}\label{lem:onlytrees}
Let $\sC$ be a class of graphs. Then \thksf\sC is a class of perfect graphs if
and only if $\sC\subseteq\sF$.
\end{lem}
\begin{proof}
If $\sC\subseteq\sF$, this is Corollary~\ref{cor:forestperfect}. Conversely, suppose there exists $H\in\sC\sm\sF$. Then $H$ has a triangle or hole, $C$. If $C$ is an odd hole, $H\notin\sP$ by the SPGT, so $\thksf H\nsubseteq\sP$. If $C$ is an even hole, then we replace any vertex $v\in C$ with
a $K_2$ $vv'$ so that an edge $vw$ of $C$ becomes $vv'w$, to give a graph $H'$ with odd hole $C'$. (See Fig.~\ref{fig:cycles}.) Then $H'\in\thksf H$, but $H'\notin\sP$, so $\thksf H\nsubseteq\sP$.
If $C$ is a triangle $uvw$, we replace it similarly with $K_2$'s $uu'$, $vv'$ so that $uvw$ becomes a 5-hole $uu'vv'w$ in a graph $H'\in\thksf H$ such that $H'\notin\sP$.
\end{proof}
Thus $\sF$ is the largest class $\sC$ for which \thksf\sC contains only perfect graphs.

However, not all perfect graphs without long holes are thick forests. See, for example see Figs.~\ref{fig:unthick} and~\ref{fig:unquasi}. We defer proof of the claims to section~\ref{sec:forbidden}.

\begin{figure}[htb]
  \noindent
  \begin{minipage}[b]{0.5\linewidth}
  \centering
  \begin{tikzpicture}[xscale=0.4,yscale=0.25,rotate=90,font={\sffamily\scriptsize}]
    \node[v] (v1) at (0,0) {f};
    \node[v] (v2) at (2,3) {h};
    \node[v] (v3) at (4,0) {e};
    \node[v] (v4) at (6,3) {g};
    \node[v] (v5) at (8,0) {d};
    \node[v] (v6) at (0,6) {c};
    \node[v] (v7) at (4,6) {b};
    \node[v] (v8) at (8,6) {a};
    \draw (v1)--(v2)--(v3)--(v4)--(v5) (v1)--(v3)--(v5) (v2)--(v4) ;
    \draw (v6)--(v7)--(v8)--(v4)--(v7)--(v2)--(v6);
  \end{tikzpicture}
  \caption{A chordal graph not in \thksf\sF}\label{fig:unthick}
\end{minipage}%
\begin{minipage}[b]{0.5\linewidth}
  \centering
  \begin{tikzpicture}[xscale=0.4,yscale=0.25,font={\sffamily\scriptsize}]
    \node[v] (v0) at (-2,3) {f};
    \node[v] (v1) at (0,0) {d};
    \node[v] (v2) at (2,3) {a};
    \node[v] (v3) at (4,0) {e};
    \node[v] (v5) at (6,3) {g};
    \node[v] (v6) at (0,6) {b};
    \node[v] (v7) at (4,6) {c};
    \draw (v1)--(v2)--(v3)--(v5) (v3)--(v5)--(v7)  ;
    \draw (v7)--(v2)--(v6) (v1)--(v0)--(v6) ;
  \end{tikzpicture}
  \caption{A bipartite graph not in $\thksf\sF$}\label{fig:unquasi}
  \end{minipage}
\end{figure}

The complexity of recognising $G\in\thksf{\sB| \sP_0}$ is thus an interesting question. This is not obviously in P, even for subclasses of $\sP_0$, for example the class $\sC_\triangle$ of chordal graphs. Stacho~\cite{Stacho} gave a polynomial time algorithm for recognition of $G\in\thksf{\sB| \sC_\triangle}$. However, we show below, in Corollary~\ref{cor:chordal}, that $\thksf{\sB| \sC_\triangle}=\thksf{\sF| \sC_\triangle}$. So recognition for $\thksf{\sB| \sC_\triangle}$ is implied by recognition for $\thksf\sF$, which we show in section~\ref{sec:recog} below.

Eschen, Ho\`{a}ng, Petrick and Sritharan~\cite{EHPS} defined a subclass of $\sP_0$ which they called
\emph{biclique separable} graphs. This is the hereditary class of graphs which are either cliques, or have a separator consisting of two disjoint cliques, and are also long hole-free. This class strictly includes \emph{weakly chordal} graphs, but does not include $\thksf\sF$. The graph on the left in Fig.~\ref{fig:incomparable} is an example. This graph is cobipartite, so in $\sF$, is not a clique and has no clique separator. The only nonadjacent cliques are the two pairs of opposite corner vertices. Removal of either of these pairs does not disconnect the graph, so it is not biclique separable. Neither is every biclique separable graph in $\sF$. The graph on the right in Fig.~\ref{fig:incomparable} is an example. It is biclique separable, in fact weakly chordal, but is not a clique, not cobipartite and has no clique separator. So it is not in $\sF$.
\begin{figure}[hbtp]
\centerline{\begin{tikzpicture}[line width=0.5pt]
    \node[v] (l) at (-1,0) {};\node[v] (r) at (1,0) {};
    \node[v] (t) at (0,1) {};\node[v] (c) at (0,0) {};\node[v] (b) at (0,-1) {};
    \draw (t)--(l)--(b)--(r)--(t) (t)--(c)--(b) (l)--(c)--(r);
  \end{tikzpicture}
  \hspace{3cm}
  \begin{tikzpicture}[line width=0.5pt]
    \node[v] (l1) at (-1,0) {};\node[v] (l2) at (-2,0) {};\node[v] (l3) at (-3,0) {};
    \node[v] (r1) at (1,0) {};\node[v] (r2) at (2,0) {};\node[v] (r3) at (3,0) {};
    \node[v] (c1) at (0,1) {};\node[v] (c2) at (0,-1) {};
    \draw (c1)--(l1)--(l2)--(l3)--(c1) (l1)--(c2)--(l3);
    \draw (c1)--(r1)--(r2)--(r3)--(c1) (r1)--(c2)--(r3);
  \end{tikzpicture}}
\caption{Biclique separable is incomparable with $\sF$ }
  \label{fig:incomparable}
\end{figure}

The class $\sP_0$ has not received much attention. Recognition is clearly in P, using the algorithms of~\cite{CLV} and~\cite{CSSS} on $\bar G$. But we can ask algorithmic questions about $\sP_0$ which might be easier to answer than for $\sP$. For example, is there a combinatorial algorithm for finding a maximum independent set in $G\in\sP_0$? More ambitiously, can we count independent sets in polynomial time for $G\in\sP_0$? We show below that this can be done for $G\in\sF$, but is \#P-complete for $G\in\sP$.

\subsection{Forbidden subgraphs}\label{sec:forbidden}
Since they are a hereditary class, thick forests can be characterised by minimal forbidden subgraphs, but there appear to be too many to be useful. Some examples are given in Fig.~\ref{fig:forb-vc}, but we will not prove these, except for the graph of Fig.~\ref{fig:forb-vc}. The others can be proved similarly. There are even infinite families of minimal forbidden subgraphs besides long holes and odd antiholes. One such family is shown in Fig.~\ref{fig:forb-inf}. We will not prove this, but it can be shown not to be a thick tree using the algorithm of section~\ref{sec:recog}, and minimal likewise, using the symmetries. We conjecture that all such infinite families of forbidden subgraphs for $\sF$ are of a similar form, two small graphs connected by a chain of triangles. 

\begin{figure}[ht]
  \[ \vcenter{\hbox{
  \begin{tikzpicture}[scale=0.707]
    \node[v] (a) at (1,3) {};
    \node[v] (b) at (3,1) {};
    \foreach[count=\c] \n in {d,e,f} {
      \node[v] (\n) at (\c,\c) {};
      \draw (a)--(\n)--(b);
    }
  \end{tikzpicture}
  }}\hspace{9mm}\vcenter{\hbox{
  \begin{tikzpicture}
    \foreach \n/\x in {c/0, e/1} \node[v] (\n) at (\x, 0) {};
    \foreach \n/\x in {g/0, f/1} \node[v] (\n) at (\x,-1) {};
    \node[v] (a) at (210:1) {};
    \node[v] (d) at ( 60:1) {};
    \draw (g)--(c)--(d)--(e)--(f)--(g)--(a)--(c)--(e);
  \end{tikzpicture}
  }}\hspace{9mm}\vcenter{\hbox{
\begin{tikzpicture}
    \foreach \n/\x in {b/-1, c/0, e/1} \node[v] (\n) at (\x, 0) {};
    \foreach \n/\x in {a/-1, g/0, f/1} \node[v] (\n) at (\x,-1) {};
    \node[v] (d) at (60:1) {};
    \draw (g)--(c)--(d)--(e)--(f)--(g)--(a)--(b)--(c)--(e);
  \end{tikzpicture}
  }}\hspace{9mm}\vcenter{\hbox{
  \begin{tikzpicture}
    \node[v] (c) at (0,0) {};
    \foreach \n/\a in {a/180, b/120, d/60, e/0, f/300, g/240} {
      \node[v] (\n) at (\a:1) {};
      \draw (c)--(\n);
    }
    \draw (g)--(a)--(b)  (d)--(e)--(f);
  \end{tikzpicture}
  }}\hspace{9mm}\vcenter{\hbox{
\begin{tikzpicture}
    \foreach \n/\a in {a/210, b/150, d/30, e/330, g/270} \node[v] (\n) at (\a:1) {};
    \node[v] (c) at (0,0) {};
    \node[v] (f) at ($(e) + (0,-1)$) {};
    \node[v] (h) at ($(a) + (0,-1)$) {};
    \draw (c)--(d)--(e)--(f)--(g)--(h)--(a)--(b)--(c)--(e)--(g)--(a)--(c)--(g);
  \end{tikzpicture}
  }} \]
  \caption{Some minimal forbidden subgraphs for thick forests}
  \label{fig:forb-vc}
\end{figure}

\begin{figure}[ht]
\centerline{
\begin{tikzpicture}[xscale=0.5,yscale=0.75]
\begin{scope}
    \node[v] (a) at (0,0) {};\node[v] (b) at (1,1) {};\node[v] (c) at (2,0) {};\node[v] (d) at (1,-1) {};
     \node[v] (e) at (4,0) {}; \node[v] (f) at (3,1) {};
     \node[v] (g) at (6,0) {}; \node[v] (h) at (5,1) {};
     \draw (a)--(b)--(c)--(d)--(a)--(c);
     \draw (c)--(e)--(f)--(c);
     \draw (e)--(g)--(h)--(e);
\end{scope}
\node at (8,0){\Large$\ldots\ldots$};
\begin{scope}[xshift=10cm]
    \node[v] (a) at (0,0) {};\node[v] (b) at (1,1) {};\node[v] (c) at (2,0) {};
     \node[v] (e) at (4,0) {}; \node[v] (f) at (3,1) {};
     \node[v] (g) at (6,0) {}; \node[v] (h) at (5,1) {};\node[v] (d) at (5,-1) {};
     \draw (a)--(b)--(c)--(a);
     \draw (c)--(e)--(f)--(c);
     \draw (e)--(g)--(h)--(e) (e)--(d)--(g);
\end{scope}
  \end{tikzpicture}}
 \caption{An infinite family of minimal forbidden subgraphs for thick forests}
  \label{fig:forb-inf}
\end{figure}
  
However, a simpler criterion can be given using forbidden \emph{coloured} subgraphs:
\begin{thm}\label{thm:forestcol}
  A graph $G$ is a thick forest if and only if it has a vertex 2-colouring which excludes
  \begin{enumerate}[itemsep=0pt,topsep=0pt,label=(\alph*)]
    \item a long hole (with any colouring),\label{item:longholes}
    \item a monochromatic $P_3$, \label{item:monochromatic}
    \item a 4-hole with the alternating colouring. (See Fig.~\ref{fig:C4col}.)\label{item:alternating}
  \end{enumerate}
\end{thm}
\begin{figure}[hbtp]
 \centerline{%
  \begin{tikzpicture}[scale=0.6]
  \begin{scope}
  \node[bl] (1) at (0,0) {};
  \node[rd] (2) at (2,0) {};
  \node[bl] (3) at (2,2) {};
  \node[rd] (4) at (0,2) {};
   \draw (1)--(2)--(3)--(4)--(1);
   \end{scope}
   \begin{scope}[xshift=6cm]
  \node[bl] (1) at (0,0) {};
  \node[bl] (2) at (2,0) {};
  \node[rd] (3) at (2,2) {};
  \node[rd] (4) at (0,2) {};
   \draw (1)--(2)--(3)--(4)--(1);
   \end{scope}
  \end{tikzpicture}}
\caption{Alternating and cobipartite colourings of a 4-hole.}
  \label{fig:C4col}
\end{figure}
\begin{proof}
Suppose $G$ is a thick forest and $H\in\thin G$. Then $G$ is a thick bipartite graph, so must have a 2-colouring satisfying \ref{item:monochromatic} by Lemma~\ref{lem:thickB}. It must satisfy~\ref{item:longholes} by Lemma~\ref{lem:holesantiholes}. Also, by Prop.~\ref{prop10}, any 4-hole must lie in a thick edge,
so it must have the cobipartite colouring (see Fig.~\ref{fig:C4col}), and cannot have the alternating colouring.
Hence \ref{item:alternating} is satisfied.

Conversely, if $G$ has a 2-colouring satisfying~\ref{item:monochromatic}, that is, a 2-subcolouring, it is a thick bipartite graph. Also $H$ cannot have a triangle or odd hole. If $H$ has a $k$-hole, $G$ has hole of size $k\leq\ell\leq2k$, by alternately using an edge of $G$ in a link and an edge of $G$ in a node, if necessary. Thus any hole must have $k\leq 4$, or~\ref{item:longholes} would be contradicted. Thus we have only to consider the case where $G$ has a 4-hole. The only 2-colouring of a 4-hole not ruled out by \ref{item:monochromatic} and \ref{item:alternating} is the cobipartite colouring. But a 4-hole of $G$ with this colouring must lie in a link since its monochromatic edges must lie in two different nodes and its dichromatic edges must lie in the edge set connecting them. Thus $H$ can have no triangles or holes and must be a forest.
\end{proof}
\begin{cor}\label{cor:chordal}
A chordal graph is 2-subcolourable if and only if it is a thick forest.
\end{cor}
\begin{proof}
  If $G$ is chordal	conditions~\ref{item:longholes} and~\ref{item:alternating} of Theorem~\ref{thm:forestcol} are automatically satisfied, since then there are no 4-holes or longer in $G$. Condition~\ref{item:monochromatic}
  is equivalent to a $G$ having a 2-subcolouring.
\end{proof}
Thus the graph class considered in~\cite{Stacho} is precisely the class of chordal thick forests.
In our notation, we have shown that $\thksf{\sB|\sC_\triangle}=\thksf{\sF|\sC_\triangle}$.

Note that, if $G$ is chordal, its links must be a cobipartite chordal graphs. These are cochain graphs, since
the complement is bipartite and has no $\bar{C_4}=2K_2$ so is a chain graph.

This forbidden coloured subgraph characterisations does not imply a polynomial time recognition algorithm. Recognising a thick bipartite graph, that is, a graph satisfying only~\ref{item:monochromatic}, is NP-complete from Lemma~\ref{lem:thickNP}. On the other hand, monochromatic forbidden subgraph characterisations do not imply polynomial time recognition when there are infinitely many forbidden subgraphs, as there are here. See, for example,~\cite{Hoang}.

Theorem~\ref{thm:forestcol} casts the recognition of thick forests as a instance of SAT if the the vertices are variables and the two colours represent \textsf{true} and \textsf{false}. Then the clauses derived from~\ref{item:monochromatic} are of type NAE (not-all-equal) 3SAT, and those derived from~\ref{item:alternating}
are of type NAE (not-all-equal) 2SAT. Thus the problem is an instance of NAE 2/3SAT. In general these problems have NP-complete decision. So this does not lead to a decision algorithm, and our approach in section~\ref{sec:recog} does not use this characterisation of thick forests.

However, we will apply Theorem~\ref{thm:forestcol} to show that the graphs of Figs.~\ref{fig:unthick} and~\ref{fig:unquasi} are not thick trees.

Consider the graph $G$ of Fig.~\ref{fig:unthick}. Suppose first that \msf{g} and \msf{h} have the same colour, \wlg red. Then \msf{a,c,d,f} must all be blue, or one of the \pth3's \msf{hga}, \msf{hgd}, \msf{ghc}, \msf{ghf} would be monocoloured red. But now \msf{b,e} must both be red, or one of the \pth3's \msf{abc}, \msf{def} would be  monocoloured blue. But now we have \pth3's \msf{bge}, \msf{bhe} monocoloured red, so this is a bad colouring.
So \msf{g} and \msf{h} have different colours, \msf{g} red and \msf{h} blue \wlg.
Now either and \msf{a} or \msf{d} must be blue. Otherwise we have a red \pth3 \msf{agd}. Assume, by symmetry, that
\msf{a} is blue. Now \msf{b} must be red or we have a blue \pth3 \msf{abh}. Then \msf{c,e} must be blue or we have a red \pth3 \msf{gbc} or \msf{bge}. But now there is a blue \pth3 \msf{che} so again we have a bad colouring.
Therefore, by Theorem~\ref{thm:forestcol}, $G$ is not a thick tree.

The graph $G$ of Fig.~\ref{fig:unquasi} has two $C_4$'s. Thus \msf{b,d} and \msf{c,e} must receive different colours. Assume, by symmetry, that \msf{b,c} are red and \msf{d,e} are blue. Then we cannot colour \msf{a}
red or we would have a red \pth3 \msf{bac} and we cannot colour \msf{a} blue or we would have a blue \pth3 \msf{dae}. Therefore, by Theorem~\ref{thm:forestcol}, $G$ is not a thick tree.

In fact, both the graphs of Figs.~\ref{fig:unthick} and~\ref{fig:unquasi} are minimal forbidden subgraphs for the class of thick forests, since it is easy to see that deleting any vertex leaves a thick forest. 

\section{Recognising a thick forest}\label{sec:recog}
We have seen in Lemma~\ref{lem:thickNP} that it is NP-complete to recognise the class \thksf\sB of thick bipartite graphs. However, this intractability does not necessarily apply to subclasses of \thksf\sB. Thus we will give an algorithm to recognise a thick forest in polynomial time. If the recognition algorithm succeeds, we will say that $G$ is \emph{accepted}, and we terminate with an explicit $H\in\thin G\cap\sF$. Otherwise, $\thin G\cap\sF=\es$, and we will say that $G$ is \emph{rejected}.
\begin{rem}\label{rem:20}
Note that we  do not guarantee to recover any particular representation of $G$ as a thick tree. In fact, $\thin G$ is usually exponentially large, as discussed in section~\ref{sec:thickgraphs}. Rather, we construct some $H\in\thin G\cap\sF$, if one exists, with failure meaning that $\thin G\cap\sF=\es$. Thus, in what follows, nodes $\bs u, \bs w, \ldots$  should not be thought of as fixed in advance, but rather as vertices of some $H\in\thin G\cap\sF$.
\end{rem}
Stacho~\cite{Stacho} gave a polynomial time algorithm to decide whether a chordal graph is 2-subcolourable. From Corollary~\ref{cor:chordal}, we know that this is the same problem as recognising that a chordal graph as a thick forest. Thus, for example, Stacho's algorithm should verify that the graph of Fig.~\ref{fig:chordalthicktree} is a thick tree. However, the algorithm is restricted to chordal thick trees, and does not appear to explicitly yield a thick tree representation. So we give a different algorithm here, which recognises a general thick forest with the same time complexity.

To decide if $G=(V,E)$ is a thick forest, it suffices to decide whether each component of $G$ is a thick tree,
so we will consider only the case where $G$ is a thick tree. Then our algorithm is based on the simple fact that the structure of a tree is precisely that it can be decomposed into disjoint trees by the removal of any vertex. Thus, if we can identify any node $\bs u$ of $G$, we can cast the recognition problem as recursively recognising the subtrees rooted at its children $\bs w_i$ $(i\in[r]$). We terminate the recursion when the subtree is a clique, a thick leaf of \thin G.

For any node $\bs u$ of a thick forest, its induced neighbourhood $G[\Nb(\bs u)]$ must be a cluster graph,
and hence $G[N[\bs u]]$ is a unipolar graph 
with $\bs u$ as its hub and satellites comprising subcliques $C_i$ of its thick neighbours $\bs w_i$ ($i\in[r]$), where $r=\dg\bs u$ in $H\in\thin G$.

To identify a $\bs u\in\thin G$, use the following. 
Suppose $G\in\thksf\sT$ and we have a clique $C$ such that $C\subseteq \bs u$ for some node $\bs u$. Then $\Nb[C]$ induces a unipolar graph $G_C$ with hub $\bs h\subseteq \bs u$.

We use the algorithm \texttt{UNIPOLAR} in section~\ref{sec:unipolar} to decompose $G_C$. A necessary ingredient is recognition of a link, a cobipartite graph. We show this in section~\ref{sec:cobip}.

While $\bs h=\bs u$ is the ``usual'' case in the unipolar decomposition, we must also deal with cases where $\bs h\subset\bs u$. We describe our method for this in section~\ref{sec:expand}.
 
To initialise the algorithm we choose any $v\in V$, which must be in some $\bs u$, and take $C=\{v\}\subseteq\bs u$. We then apply this recursively to the subtrees rooted at $\bs u$'s neighbours $\bs w_i$ ($i\in[r]$), using the clique $C_i\subseteq\bs w_i$ from the cluster graph $G[\Nb(\bs u)]$ as seed to determine the $\bs w_i$. The recursion terminates when a subtree is simply a node. The totality of nodes that appear in the algorithm are then those of the tree \thin G, and we accept $G$. If any failure occurs in the algorithm, we stop, concluding that $G$ is not a thick tree and reject it.

This essentially completes the description of the algorithm. However, there is a remaining difficulty, due to subgraphs of $G$ which can be viewed as links in more than one way. In the above algorithm, this can result in a false rejection of $G$. We consider how to recognise and resolve this problem in section~\ref{sec:adjust}.

Finally, in section~\ref{sec:time}, we show that the time complexity of the algorithm is $\cO(mn)$.

\subsection{Recognising a cobipartite graph}\label{sec:cobip}

\begin{lem}\label{lem:cobiprecog}
A cobipartite graph $G=(V_1\cup V_2,E)$, with possible pre-assignments of vertices to $V_1$ and $V_2$, is recognisable in $\Theta(m)$ time.
\end{lem}
\begin{proof}
The following simple algorithm recognises a cobipartite graph by recognising its complement as a bipartite graph.

\texttt{COBIPARTITE}
\begin{enumerate}[topsep=0pt,itemsep=0pt,label=(\arabic*)]
\item Form the complementary graph $\bar{G}=(V_1\cup V_2,\bar{E})$ in $\cO(n^2)$ time. 
\item Determine the connected components $B_1,\ldots, B_k$ of $\bar{G}$ in $\cO(n^2)$ time. 
\item Check whether any $B_i$ is not bipartite in $\cO(n^2)$ time. If so, reject $G$. Otherwise, let
$B_i=B_{i1}\cup B_{i2}$ ($i\in[k]$) be the bipartitions. 
\item Determine $V_1$ and $V_2$ by assigning $B_{i1}, B_{i2}$ to $V_1$ and $V_2$ for $i\in[k]$, taking into account pre-assignments. If these conflict, reject $G$. Otherwise put the larger of $B_{i1}, B_{i2}$ in $V_2$ or choose arbitrarily if $|B_{i1}|=|B_{i2}|$.\label{cob:step4}
\end{enumerate} 
Maximising the size of $V_2$ in step~\ref{cob:step4} is a choice we make for use in \texttt{UNIPOLAR} below. In general, ignoring pre-assignments, there are $2^k$ ways of assigning the bipartitions.
 
To complete the proof we need only show that $n^2=\Theta(m)$.
Let $n_1=|V_1|$, $n_2=|V_2|$, so $n=n_1+n_2$ and $|V_1:V_2|\leq n_1n_2$. Then
\[ \tfrac14 n(n-1)\leq \tfrac12n_1(n_1-1)+\tfrac12n_2(n_2-1) \leq m \leq \tfrac12 n_1^2+\tfrac12n_2^2+n_1n_2\leq n_1^2+n_2^2\leq n^2,\]
where we have used the convexity of $x(x-1)$ with $x=n/2$ in the first inequality, and $2n_1n_2\leq n_1^2+n_2^2$ in the fourth. Thus $m=\Theta(n_1^2+n_2^2)=\Theta(n^2)$, and so any algorithm recognising $G$ must be $\Theta(m)=\Theta(n^2)$. 
\end{proof}
Note that we may be able to improve the time bound if there are sufficiently many pre-assignments, but we will not
pursue this here.

\subsection{Recognising a unipolar graph}\label{sec:unipolar}
We first show how we can recognise a unipolar graph, a thick star, with hub $\bs h$ and satellites $\bs b_1,\bs b_2,\ldots\bs b_s$. From Lemma~\ref{lem:modelunipol} we know that a unipolar decomposition is not necessarily unique, so we make the following assumption:\\[3pt]
($*$) As many vertices as possible are moved from the hub into one satellite.\\[3pt]
Let $\bs b_{s+1}\gets\es$.  Then, if $\Nb(v)=\bs h\cup\bs b_i$ for any $i\in[s+1]$, we can move $v$ from $\bs h$ into $\bs b_i$ by Lemma~\ref{lem:goodmove}. Moving all such $v$ from $\bs h$ into $\bs b_i$ will ensure that ($*$) is satisfied at the end of the algorithm. We can only choose one such $i$ by  by Lemma~\ref{cor:neighbour}. If now $\bs b_{s+1}\neq\es$, set $s\gets s+1$, otherwise delete $b_{s+1}$. The purpose of ($*$) is that $\bs h$ is minimal. This is useful in section~\ref{sec:expand} below, since it means that every vertex in $\bs h$ must be in the hub.
 
Unipolarity can be recognised by existing methods~\cite{EW,McYo,Tysh}, but we will now give a simple $\cO(mn)$ time algorithm using the same ideas as for a general thick tree. The approach is closest to that in~\cite{Tysh} in that its uses only vertex neighbourhoods and testing for cobipartiteness, which we can check using the algorithm \texttt{COBIPARTITE} given in section \ref{sec:cobip} above.

The justification is given in the bullet point comments.

\texttt{UNIPOLAR}
\begin{enumerate}[topsep=0pt,itemsep=0pt,label=(\arabic*)]
  \item $i\gets1$, $V_1\gets V$.\label{step1}
  \item $G_i\gets G[V_i]$. Choose $u_i,w_i\in V_i$ such that $u_iw_i$ is a nonedge of $G_i$. If no $u_i,w_i$ exist, $\bs h\gets G_i$ and stop, we have found a unipolar decomposition.\label{step2}\\[2pt]
      $\bullet$ $G_i$ is a clique, so must be $\bs h$. Otherwise, the nonedge $u_iw_i$ is used in \ref{step6}.
  \item Let $U_i\gets\Nb[u_i]\cap V_i$, $\bar{U}_i\gets V\sm U_i$.\label{step3}\\[2pt]
       $\bullet$ Loop initialisation. The node $\bs u_i$ of $G$ containing $u_i$ (either $\bs h$ or $\bs b_i$) satisfies $\bs u_i\subseteq U_i$. If $\bs u_i=\bs b_i$,  $U_i$ is an initial estimate of $\bs b_i$.
  \item Test whether $G[U_i]$ is cobipartite, with $U_i=A_i\cup B_i$, where every vertex $v\in A_i$ has a nonneighbour in $B_i$ or a neighbour in $\bar{U}_i$, and $B'_i=\Nb(B_i)\cap\bar{U}_i$ is a (possibly empty) clique complete to $B_i$. \label{step4} \\[2pt]
      $\bullet$ Otherwise can we move $v$ from $A_i$ to $B_i$, minimising the size of $A_i$. Note that we can have $B_i=\es$ if $G[U_i]$ is a clique. 
  \item If so, test whether $G[\bar{U}_i]$ is a cluster graph. If so, $\bs h\gets A_i$, $\bs b_i\gets B_i\cup B'_i$ and the satellites are given by $\bs b_i$ and $G[\bar{U}_i\sm B']$. Stop, we have a unipolar decomposition.\label{step5}\\[2pt]
       $\bullet$ $U_i$ is a cobipartite separator in $G$. It cannot have vertices in two satellites, or else it is not cobipartite. Also $\bs h\subseteq U_i$ or $G[\bar{U}_i]$ would not be a cluster graph. So $U_i\subseteq\bs h\cup\bs b_i$. Now if $u_i\in\bs h$, $\bs h\subset U_i$, so $U_i=\bs h\cup B_i$, where by ($*$) we must have $A_i=\bs h$ and $B_i\subseteq \bs b_i$. Since $B'_i$ is a clique, $G[\bar{U}_i]$ is a cluster graph if $B'_i\neq\es$. But we can expand $B_i$ to $\bs b_i$ by adding $B'_i$, since it is complete to $B_i$. Now $G[\bar{U}_i\sm B'_i]$ gives the other satellites $\bs b_j$ ($j\in[s\sm i]$). If $u_i\in\bs b_i$, $\bs h\subseteq U_i$ implies that $u_i$ is complete to both $\bs h$ and $\bs b_i$, and we can move $u_i$ into $\bs h$, and follow the reasoning above with $B'_i=\es$.    
  \item Set $u_i\gets w_i$ and reset quantities as in step~\ref{step3}.\label{step6}\\[2pt]
       $\bullet$  $G[U_i]$ is a general unipolar subgraph of $G$, so $u_i\in\bs h$. Any nonedge in $G$ has at least one vertex in a satellite, since $\bs h$ is a clique. So $u_i\in\bs h$ implies $w_i\in\bs b_i$, and $u_i\gets w_i$ implies $G[U_i]$ is a subgraph of $\bs h\cup\bs b_i$.\label{step5}
  \item For each $v\in U_i$, let  $W(v)\gets\Nb(v)\cap \bar{U}_i$ and determine whether $W(v)$ is a clique. If not, we set $U_i\gets U_i\sm v$, removing $v$ from $U_i$.\label{step7}\\[2pt]
      $\bullet$ If $v \in \bs b_i$, $W(v)\subseteq \bs h$, so is a clique. Thus any $v$ which is removed from $U_i$ must be in $\bs h$, not in $\bs b_i$. Any remaining $v\in U_i\cap \bs h$ must be such that $W(v)\cup U_i=\bs h\cup B_i$, where $B_i\subseteq\bs b_i$, since $v\in\bs h$. Note that $\bs h$ here may depend on $v$ unless $\bs h$ is unique.
  \item If $U_i$ is a clique and $\Nb(U_i)$ is a clique, set $\bs b_i\gets U_i$ and $V_{i+1}\gets V_i\sm U_i$. Let $i\gets i+1$ and return to step~\ref{step2}.\label{step8}\\[2pt]
      $\bullet$ This will be true if and only if, $U_i\cap\bs h=\es$, in which case $U_i=\bs b_i$.
  \item Determine a $v\in U_i$, such that $W(v)\gets W(v)\cup U_i$ gives a unipolar decomposition as in steps~\ref{step4} and \ref{step5} above. If so, stop with this unipolar decomposition.\label{step9}\\[2pt]
      $\bullet$ Then $W(v)=\bs h\cup B_i$, where $B_i\subseteq \bs b_i$, so $W(v)$ is a cobipartite separator in $G$ as in steps~\ref{step4} and \ref{step5}.
  \item Stop, $G$ is not a unipolar graph.\label{step10}\\[2pt]
  $\bullet$ Since we have not stopped with a unipolar decomposition, $G$ is not unipolar.
\end{enumerate}

It is not difficult to see that \texttt{UNIPOLAR} has an $\cO(ms)$ time bound,
where $s$ is the number of satellites,
since we consider at most $2s$ vertices $u_i,w_i$. For each such vertex, the most time-consuming operation are the $\cO(m)$ time cobipartite decompositions. Since $s<n$, this is clearly an $O(mn)$ time algorithm. However it appears that the time complexity can be reduced using a more careful implementation and analysis, possibly to $\cO(m)$. See~\cite{McYo} for a different algorithm. For example, if we find $v\in\bs h$ during the identification of a satellite, this information can be used when identifying further satellites. But we will not pursue this here.

As an example, consider the graph in Fig.~\ref{fig:unipolar}. Suppose we choose nonedge $u_1w_1\gets\sf{gh}$, giving $U_1=\Nb[\sf g]\gets\bsf{c,d,e,f,g,i,j}$. Then $G[U_1]$ is not cobipartite, so $\sf g\in\bs h$ and $u_1\gets \sf h$. Then $U_1=\Nb[\sf h]\gets \bsf{f,h,i,j}$. Now $\sf f$ has nonedge neighbours in $\bar{U}_1$, $\sf{c,d}$ for example, so $U_1\gets\bsf{h,i,j}$ which has neighbours only in the clique $\bsf{f,g}$ in $\bar{U}_1$. So $\bs b_1\gets \bsf{h,i,j}$ and $\sf{f,g}\in\bs h$. Now $V_2\gets\bsf{a,b,c,d,e,f,g}$. Suppose now $u_2w_2\gets\sf{ed}$,
so $U_2\gets\bsf{a,b,e,f,g}$. Now $G[U_2]$ is cobipartite, with $A_2\gets \bsf{e,f,g}$ and $B_2\gets \bsf{a,b}$,
and $G[\bar{U}_2]$ is a cluster graph with cliques $\bsf{c,d}$, $\bsf{h,i,j}$. Now $B'_2=\es$, so $\bs h\gets \bsf{e,f,g}$, $\bs b_1\gets\bsf{h,i,j}$, $\bs b_2\gets\bsf{a,b}$ and $\bs b_3\gets\bsf{c,d}$ is a unipolar decomposition.
\begin{figure}[htb]
  \centering
  \begin{tikzpicture}[xscale=1.5,yscale=0.75,line width=0.5pt,font={\sffamily\scriptsize}]
    \node[v] (a) at (0.75,0) {a};
    \node[v] (b) at (0.75,1) {b};
    \node[v] (c) at (0.75,2.5) {c};
    \node[v] (d) at (0.75,3.5) {d};
    \node[v] (e) at (2,1) {e};
    \node[v] (f) at (2,2) {f};
    \node[v] (g) at (2,3) {g};
    \node[v] (h) at (3,1) {h};
    \node[v] (i) at (3,2) {i};
    \node[v] (j) at (3,3) {j};
    \draw (a)--(e)--(f) (g)--(d)--(c)--(f)--(b)--(a) ;
    \draw (h)--(i) (h)--(f) (b)--(e) (c)--(g) ;
    \draw (g)--(j)--(f)--(i)--(g) (f)--(g) (i)--(j) ;
    \draw (g) to[bend right=20] (e) (j) to[bend left=20] (h);
  \end{tikzpicture}
\caption{Unipolar graph $G$}\label{fig:unipolar}
\end{figure}

\subsection{Identifying a node}\label{sec:expand}

A sufficient condition for $\bs u$ to be a node in a thick tree $G$ is that it $\bs u$ is a clique and $G[V\sm\bs u]$ is a thick forest with $s$ trees rooted at $\bs u$'s neighbours $\bs w_i$ $(i\in[s])$. But to find such $\bs u$ we use weaker necessary conditions: that $G[\Nb[\bs u]]$ is a unipolar graph with satellites $C_i\subseteq\bs w_i$, and that the $C_i$ are not connected in $G$.

So, at a general step, we have a clique $C\subseteq \bs u$ for some node $\bs u$ in $G$ and we wish to find $\bs u$. To initialise the process, we may choose any $v\in V$. This must be in some node $\bs u$. Then $C=\{v\}\subseteq\bs u$, as required.

Since $C\subseteq \bs u$, we construct the graph $G_C=G[V_C]$, where $V_C=\Nb[C]$.
Now any clique in $G$ containing $C$ is also a clique in $G_C$ and so $\bs u \subseteq V_C$. 
Thus $G_C$ is a unipolar subgraph of $G_{\bs u}$ with hub $\bs h$ and satellites $\bs b_1, \bs b_2, \ldots, \bs b_s$, where $\bs b_i$ is contained in a thick neighbour $\bs w_i$ of $\bs h$, for $i\in[r]$. 
Also, since $C$ is a hub for $G_C$, and $C\subseteq\bs u$, and $\bs u\subseteq V_C$, $\bs u$ is a hub for $G_C$.

Note that $\bs u$ is a \cs in $G$ if $\bs u$ is not a thick leaf, 
and hence also in $G_C$. Thus any clique in $G_C$ which includes $\bs u$ is a \cs in $G$. However, $G_C$ also has vertices from neighbours $\bs w_i$ $(i\in[s])$ of $\bs u$ in $H=\thin G$, 
where $s\leq r=\dg_H(\bs u)$. Since $\bs w_i$ is a clique, these also induce cliques in $G_C$. Also $\bs u$ may separate $G$ into two graphs, $G_L=G[V_L]$ and $G_R=G[V_R]$, connected by $\bs u$. We will take $V_C\subseteq V_L$ and $V_R=V\sm V_L$. Thus $V_L\neq\es$, but $V_R=\es$ is possible if $C\neq\bs u$.
Now $V_L\sm V_C$ can be discovered by searching in $G$ starting from $\Nb(C)$, and hence $V_R$ can be identified. This preprocessing requires $\cO(m)$ time.

We apply the algorithm \texttt{UNIPOLAR} to $G_C$, giving a hub $\bs h$ such that $C\subseteq\bs h\subseteq\bs u \subseteq V_C$. The only modification is that a vertex $v$ cannot be in a satellite if it  has a neighbour in $V_R$ or is in $C$. This is required for $\thin G$ to be triangle-free and to ensure that $C\subseteq\bs h$. Then, at termination, $\bs h$ is a hub for $G_C$ and the only edges from $V_L$ to $V_R$ are from $\bs h$. Thus, if we check that $G[\Nb(\bs h)]$ is a cluster graph with cliques $C_i$ $(i\in[s])$, we can take $\bs u\gets\bs h$.
We check that all the $C_i$ are cliques in $\cO(m)$ time, and are  disconnected in $G$, which can be done with a breadth-first search starting from the $C_i$ in $\cO(m)$ time. The time needed to identify $\bs u$ is $\cO(mn)$, from the time bound for determining $V_L,V_R$, the time bound for \texttt{UNIPOLAR} and the time bound for checking that $\Nb(\bs u)$ is a cluster graph. This is dominated by \texttt{UNIPOLAR}.

Note that we will only obtain a model in case~\ref{normalcase} of Lemma~\ref{lem:modelunipol} from the way $G_C$ is constructed. Thus, if $G$ is chordal, we claim $\bs u\gets\bs h$, since two satellites in $G_C$ cannot be connected via a vertex not in $V_C$. This case was considered by Stacho~\cite{Stacho}. Otherwise $G$ contains a hole of size at least four, so is not chordal. But the existence of 4-holes in $G$ can mean that the hub $\bs h$ from \texttt{UNIPOLAR} satisfies only $\bs h\subset\bs u$. That is, we are in case~\ref{othercase} of Lemma~\ref{lem:modelunipol}. We must deal with this below, but first consider the chordal graph of Fig.~\ref{fig:chordalthicktree} as an example, taking $C=\bsf{z}$. Thus $V_C=\bsf{t,v,x,y,z}$, $V_R=\es$ and $G_C$ comprises three triangles with common edge \msf{yz}. So \texttt{UNIPOLAR} gives $\bs h=\bsf{y,z}$ and satellites $\bsf{t}$, $\bsf{v}$ and $\bsf{x}$. Then $\Nb(\bs h)$ gives $C_1=\bsf{s,t}$, $C_2=\bsf{u,v}$ and $C_3=\bsf{w,x}$. These are disconnected in $G$, and so we may take $\bs u=\bsf{y,z}$. Continuing this into the three subtrees headed by the $C_i$ $(i\in[3])$ gives the expected horizontal $K_2$'s as a further 12 nodes. 

Now, in the presence of 4-holes, two of the satellites produced by \texttt{UNIPOLAR} may be connected as in Fig.~\ref{fig:othercase} of Lemma~\ref{lem:modelunipol}. Then we cannot claim $\bs u\gets\bs h$. However, since 4-holes can only occur in the links of a a thick tree, two of the vertices of any 4-hole must be in $\bs u$ and two must be in some $\bs w_i$. Thus the paths connecting satellites in $G_C$ must be $P_3$'s having one endpoint in a satellite $\bs b_k\subseteq \bs u\sm\bs h$, and the other in some $\bs b_j$ ($j\neq k$), with a mid vertex not in $V_C$. If the shortest path is not a $P_3$, $G$ has a hole of size 5 or more, so can be rejected.

So we search for these $P_3$'s, using a breadth-first search starting from the $\bs b_i\subseteq V_C$ given by \texttt{UNIPOLAR}, in $\cO(m)$ time. We could now simply add these paths to $G_C$, giving a graph $G'_C$, and repeat \texttt{UNIPOLAR} on $G'_C$, but there is a more efficient alternative.

We sort the endpoints of the $P_3$'s by the $\bs b_i$ in which they appear. If at least three $\bs b_i$'s contain endpoints, this must identify a unique $\bs b_k$ such that $\bs h'= \bs h\cup\bs b_k$ is a node. Otherwise, $G$ is not a thick tree since \thin G contains a 4-cycle which is not in a link. Now we test whether $\Nb(\bs h')$ induces a cluster graph in $G$. If not, $\Nb(\bs h')$ is not unipolar, so $G$ is not a thick tree. If only two $\bs b_i$'s contain endpoints, there are two alternatives for $\bs b_k$, so we simply try both as above. If either succeeds we continue with this $\bs u\gets \bs h'$. If both fail,  again $G$ is not a thick tree.

As an illustration, consider the graph of Fig.~\ref{fig:expand2} with $C=\msf b$. Then $G_C$ is as shown, and \texttt{UNIPOLAR} gives $\bs h=\bsf{b}$ and satellites $\bs b_1=\bsf{a}$, $\bs b_2=\bsf{c}$ and $\bs b_3=\bsf{e}$. These are connected in $G$ by $P_3$'s $\msf{ade}$ and $\msf{cfe}$. Both paths have endpoint $\msf{e}$ in $\bs b_3$ so we have $\bs u\subseteq\bs h\cup\bs b_3=\bsf{b,e}$. Now $\msf{e}$ can be moved into $\bs h$, giving $\bs h\gets\bsf{b,e}$. Now $\Nb(\bs h)=\bsf{a,d}\cup\bsf{c,f}$ induce a cluster graph in $G$.

However, suppose, we had chosen $C=\msf a$. Then $G_C$ is the $P_3$ $\msf{dab}$. Then \texttt{UNIPOLAR} gives $\bs h=\bsf{a}$ and satellites $\bs b_1=\bsf{b}$, $\bs b_2=\bsf{d}$. These are connected in $G$ by a $P_3$, $\msf{bed}$, which has endpoints $\msf{b}\in\bs b_1$ and $\msf{d}\in\bs b_2$. So we have two candidates $\bs u_1\subseteq\bs h\cup\bs b_1=\bsf{a,b}$ and $\bs u_2\subseteq\bs h\cup\bs b_2=\bsf{a,d}$. Now $\Nb(\bs u_1)=\bsf{d,e}\cup\bsf{c}$, which is not a cluster graph in $G$, since it is connected by the path $\msf{cfe}$. Thus $\bs u\neq\bs u_1$, so we try $\bs u_2=\bsf{a,d}$. Then $\Nb(\bs u_2)=\bsf{b,e}\cup\bsf{g}$, which induces a cluster graph in $G$, since its two components are not connected. So we may take $\bs u\gets\bs u_2$ and continue the algorithm.

\begin{figure}[htb]
  \centering
  \begin{tikzpicture}[xscale=0.85,yscale=0.5,font=\sffamily,font={\sffamily\scriptsize}]
  \begin{scope}
    \node[v] (d) at (0,0) {d};
    \node[v] (e) at (2,0) {e};
    \node[v] (f) at (4,0) {f};
    \node[v] (a) at (0,-2) {a};
    \node[v] (b) at (2,-2) {b};
    \node[v] (c) at (4,-2) {c};
    \node[v] (g) at (0,2) {g};
    \node[v] (h) at (4,2) {h};
    \draw (a)--(b)--(c) (d)--(e)--(f);
    \draw (a)--(d)--(g) (e)--(b) (c)--(f)--(h);
    \node[empty] at (1.8,-3.5){\mbox{\normalsize $G$}};
  \end{scope}
  \begin{scope}[xshift=9cm,xscale=0.75]
    \node[v] (a) at (0,-2) {a};
    \node[v] (b) at (2,-2) {b};
    \node[v] (c) at (4,-2) {c};
    \node[v] (e) at (2,0) {e};
    \draw (a)--(b)--(c) (e)--(b);
    \node[empty] at (1.8,-3.5){\normalsize$G_C$};
  \end{scope}
  \end{tikzpicture}\caption{Graph $G$ with $G_C$ for $C=\{\msf{b}\}$}\label{fig:expand2}
\end{figure}

We now take $C_i\subset\bs w_i$ for the components $C_i$ $(i\in[r])$ of the cluster graph $G[\Nb(\bs u)]$ and use the same process to recognise the $r$ subtrees rooted at the $\bs w_i$, terminating when all subtrees are cliques.

\subsection{Fixing a path}\label{sec:adjust}
There is a remaining difficulty with the recognition algorithm. As described 
above, it could result in an incorrect rejection. By ($*$), the algorithm of 
section~\ref{sec:unipolar} moves as many vertices as possible into one of its 
satellites. Thus we may mistakenly locate vertices from a node in one of its 
children and this cannot be recognised at that point in the algorithm. Hence 
we can misidentify thick vertices and edges, which can lead to further 
misidentifications. 

The generic situation is shown in Fig.~\ref{fig:adjust} by the loose triangle $(\msf{abc})$. As in section~\ref{sec:loose}, the heavy lines indicate full links, and the others are arbitrary links. Thus $\msf{ab}$ and $\msf{bc}$ are possible thick vertices,  and $(\msf{abc})$ represents a thick edge, which may be $\msf{ab:c}$ or $\msf{a:bc}$, as discussed in section~\ref{sec:loose}. We may also have loose 4-cycles, like $(\msf{efhg})$ in Fig.~\ref{fig:adjust}, with two full links representing the thick vertices of a thick edge $\msf{ef}:\msf{gh}$. As in section~\ref{sec:loose}, the 4-cycle can have chords, but we assume these links are not full, or the 4-cycle represents two thick edges. There is no need consider longer cycles since thick edges contain only triangles and 4-holes.

\begin{figure}[htb]
  \centering
  \begin{tikzpicture}[xscale=0.8,yscale=1,font={\sffamily\scriptsize}]
    \node[v] (a) at (4,6) {a};
    \node[v] (b) at (3,5) {b};
    \node[v] (c) at (5,5) {c};
    \node[v] (d) at (7,5) {d};
    \node[v] (e) at (6,4) {e};
    \node[v] (f) at (8,4) {f};
    \node[v] (g) at (6,3) {g};
    \node[v] (h) at (8,3) {h};
    \draw[ultra thick] (a)--(b)--(c)--(d)--(e) --(f) (h)--(g) ;
    \draw[thin] (a)--(c)--(e) (d)--(e)--(g) (f)--(h);
  \end{tikzpicture}\caption{Fixing a path}\label{fig:adjust}
\end{figure} 

A loose triangle with exactly one full link can be a thick edge in one way, so there can be no misidentification. If there are two full links, the triangle can be a thick edge in two ways, and if there are three, the triangle is a thick vertex.  

Thus confusion between thick vertices and edges can occur. We can view this as vertices being moved wrongly from a thick vertex into one of its neighbours. This can then propagate into one of the subtrees, and only one, by Cor.~\ref{cor:oneneighbour}. Thus misidentifications can only occur along a path in the thick tree. And, since the thin graph must be a tree, the path must terminate at or before a thick.

This becomes problematic only when one of the nodes of the loose triangle is also in a loose 4-cycle, as with $(\msf{efhg})$ in Fig.~\ref{fig:adjust}. Here we must identify $\msf{de:gh}$ as a thick edge or the graph is not a thick tree. However, this can result in a node being contained in two thick vertices, so we do not have a thick tree. We will call this a \emph{clash}. In Fig.~\ref{fig:adjust}, for example, $\msf{e}$~is in both $\msf{de}$ and $\msf{ef}$, so is a clash.

To resolve this, we must remove the clashing node from one of the thick vertices. in Fig.~\ref{fig:adjust} we cannot remove $\msf e$ from $\msf{ef}$ or we have a 4-cycle which is not in a thick edge, implying rejection of $G$. So we must remove $\msf d$ from $\msf{de}$, if possible. We can do this here since $cd$ is a full link, giving $\msf{cd:e}$ a thick edge, and \msf{de} no longer a thick vertex. We identify \msf{cd} as follows. Let $\Nb^-$ denote neighbourhood upwards in the tree. Since $\Nb^-(\msf e)=\msf c\cup \msf d$, we must move $\msf d$ into the thick vertex $\msf{cd}$. If $\msf{cd}$ is not a full link, then $G$ is not a thick tree, since we cannot resolve the clash. But resolving this clash may now cause a clash with thick vertex $\msf{bc}$. Removing \msf c from \msf{bc} will then give the thick vertex $\msf{ab}$ and so on. The path of nodes and links involved is identified by the algorithm of section~\ref{sec:recog}. This path of clashes cannot go beyond the root, since that is where the first misidentification could  occur. So this is a kind of backtracking but, since there is no branching, it does not significantly worsen the time complexity of the algorithm.  

We accept $G$ if this terminates with no clash, and reject $G$ if there is a clash which cannot be resolved because some upward neighbourhood is not a clique. Encountering a 4-cycle now results in rejection. We cannot move the clashing node upwards in the schematic, since its neighbourhood is a $P_3$. This occurs if both nodes in a thick vertex have links to a third. For example, if there was a node $\msf k$ in Fig.~\ref{fig:adjust} so that $\msf{bck}$ formed a triangle, making $\msf{cd}$ a thick vertex gives a fatal clash at $\msf c$. In fact, such a fatal clash reveals a forbidden subgraph from the class depicted in~Fig.~\ref{fig:forb-inf}. 

This problem can occur on several paths in the thick tree, even if $G$ is accepted. But each such path is bounded by different loose triangles and 4-cycles in the schematic, so these paths must all be disjoint. Thus the total work to perform all fixes is $\cO(m)$.

To illustrate this, consider the simple graph of Fig.~\ref{fig:hiddenV1}. The vertices are 1-cliques, and therefore all links are full, as shown. 

\begin{figure}[htb]
  \centering
  \begin{tikzpicture}[xscale=0.8,yscale=1,font={\sffamily\scriptsize}]
    \node[v] (a) at (4,6) {a};
    \node[v] (b) at (1,5) {b}; 
    \node[v] (c) at (3,5) {c};
    \node[v] (d) at (5,5) {d};
    \node[v] (e) at (7,5) {e};
    \node[v] (f) at (2,4) {f};
    \node[v] (g) at (6,4) {g};
    \node[v] (h) at (8,4) {h};
    \node[v] (i) at (1,3) {i};
    \node[v] (j) at (3,3) {j};
    \node[v] (k) at (6,3) {k};
    \node[v] (l) at (8,3) {l};
    \draw[ultra thick] (a)--(c)--(b) (a)--(d)--(e) (d)--(c) (g)--(h)--(l)--(k)--(g) ;
    \draw[ultra thick] (b)--(f)--(c) (d)--(g)--(e) (f)--(i)--(j)--(f);
  \end{tikzpicture}\caption{Example: Fixing a path}\label{fig:hiddenV1}
\end{figure}
Suppose we start with $C=\bsf{a}$, then \texttt{UNIPOLAR} gives $\bs u=\bsf{a}$, and $C_1=\bs w_1=\msf{cd}$. The node $\msf{cd}$ then has children $\msf{bf}$ and $\msf{eg}$. In the subtree rooted at $\msf{bf}$, we have a leaf $\msf{ij}$, so this subtree is accepted. The node $\msf{eg}$ has two children $\msf{h}$ and $\msf{k}$  but these are connected by the $\textrm{P}_3$ $\msf{klh}$. The 4-cycle $\msf{gklh}$ can be resolved as a thick edge $\msf{gh}:\msf{kl}$ or $\msf{gk}:\msf{hl}$, but now $\msf g$ is a clash with $\msf{eg}$. So we must move $\msf g$. Now $\Nb^-(\msf g)=\msf d \cup \msf e$, which can be a node $\msf{de}$. But now $\msf{d}$ clashes with $\msf{cd}$. So we must move $\msf d$ from $\msf{cd}$ to give a new root $\msf{ac}$, and we are done. So we accept $G$ as a thick tree with nodes, $\msf{ac}$, $\msf{bf}$, $\msf{ij}$, $\msf{de}$, $\msf{gh}$ and $\msf{kl}$. The path of modifications changes the nodes $\msf{a, cd, eg}$ to  $\msf{ac, de, gh}$.

\subsection{Running time analysis}\label{sec:time}
We now determine the time and space complexity of the algorithm. The space required is clearly $\cO(m)$. 
The time to identify a node is $\cO(mn)$. We claim that the recognition algorithm has time complexity $T(m,n)=\cO(mn)$. This is easily seen to be true for the nodes. The recursion then splits $G$ into subtrees, which we will assume have $n_i$ vertices and $m_i$ edges ($i\in[r]$).

Then, since $m_i+n_i<m+n$ for $i\in[r]$, the time complexity is bounded by
\begin{align*}
  T(m,n) &=\, \sum_{i=1}^{r} T(m_i,n_i)+\cO(mn), \ \mbox{to determine a root,}\\
         &=\,\cO\Big( \sum_{i=1}^{r} m_in_i\Big) +\cO(mn), \ \mbox{by induction,}\\
         &=\, \cO\Big(\sum_{i=1}^{r}m_i\sum_{j=1}^{r}n_j+mn\Big)\\
         &=\, \cO (mn), \ \mbox{since }\sum_{i=1}^{r}m_i\leq m,\ \sum_{j=1}^{r}n_j\leq n.
\end{align*}
Thus the overall time complexity of recognising a thick forest is $\cO(mn)$.

\section{Clique cutset decomposition}\label{sec:CCD}
An important property of thick forests is that they have a \emph{\CCD}, given by the internal nodes. 
In general, a \CCD gives a tree whose internal vertices correspond to clique separators and leaves called \emph{atoms}, which are either cliques or subgraphs with no clique separator. See~\cite{BPS,BPV} for definitions and some applications of \CCD, and~\cite{BPV,Tarjan} for efficient algorithms. Using these algorithm, the \CCD of any graph $G$ can be done in $\cO(mn)$ time, the same as our time bound for recognising a thick forest.

It is easy to see that \CCDs have the following property.
\begin{prop}\label{prop:ccd}
  Let $\sC$ be a hereditary graph class, and let $\sC_0$  be its subclass of graphs which are either cliques or have no clique separator. Then, in any \CCD of $G\in\sC$, all atoms must be in $\sC_0$.
\end{prop}
Note that $\sC_0$ is not hereditary in general, since removing a vertex can produce a graph with a clique separator. For example, if $\sC_0$ is the class of long holes, removing any vertex gives a path which has a clique separator.

For a thick tree all atoms are either cliques or cobipartite graphs. However, \CCD does not necessarily recover the thick tree structure, for two reasons. Firstly, a thick tree may have clique separators which are not nodes. Secondly, the cliques in a \CCD do not need to be disjoint, as they do in a thick tree.  

Thus we cannot use \CCD to recognise a thick forest. However, the fact that it decomposes the graph is very useful in algorithmic applications, since it can reduce the problem on $\sC$ to that on $\sC_0$. For example, we will use it in sections~\ref{sec:indsets} and \ref{sec:colouring} to give algorithms counting independent sets and colourings for graphs in $\sF$.

This leads us define a larger graph class $\sQ$ by simply requiring that all atoms of a \CCD are either cliques or cobipartite graphs with no clique separator. We will call $\sQ$ the class of \emph{quasi} thick forests. Clearly $\thksf\sF\subseteq\sQ$. But, from Prop.~\ref{prop:ccd}, we see that if $G\in\sQ$, then any \CCD of $G$ proves membership in the class, which is not true for $\sF$. A chordal thick edge is a cochain graph, which decomposes into two overlapping cliques, which is not permitted for a thick tree. 

A graph in $\sQ$ can be recognised in $\cO(mn)$ time using \CCD and the recognition of cobipartite graphs from section~\ref{sec:cobip}. Quasi thick forests include the class $\sC_\triangle$ of chordal graphs, where the atoms of the \CCD are all cliques. A graph is chordal \ifff it has a \CCD with only cliques as atoms~\cite{Gavril}. 

Clearly not all thick forests are chordal, but neither are all chordal graphs thick forests.  The graph shown in Fig.~\ref{fig:unthick} is an example of a chordal graph which is not a thick tree, as we proved in section~\ref{sec:forbidden}. Thus $\sF\cup\sC_\triangle\subseteq\sQ$. A non-chordal quasi thick tree may also fail to be a thick tree. See Fig.~\ref{fig:unquasi} for an example. This graph is in $\sQ$ since the clique separator $\bsf a$ decomposes it into two cobipartite graphs $\msf{abfd}$ and $\msf{acge}$, but we proved in section~\ref{sec:forbidden} that it is not in $\sF$. Thus $\sF\cup\sC_\triangle\subset\sQ$.

We can use properties of \CCD to prove facts about $\sC$ from those for~$\sC_0$. These are stated without proof in~\cite{BPS}, but follow easily from Prop.~\ref{prop:ccd}.
\begin{prop}\label{prop10}
  Let $H$ be a graph with no clique separator. Then any induced copy of $H$ in a graph $G$ must be contained in an atom of any \CCD of $G$.
\end{prop}
We can use Prop.~\ref{prop10} to prove perfection of $\sQ$.
\begin{prop}\label{prop20}
A hole or (long) antihole has no clique separator.
\end{prop}
\begin{proof}
Clearly a hole has no clique separator. So let the antihole be $\bar{C}$, where $|C|=k\geq5$. Any vertex in $\bar{C}$ has only two non-neighbours, so any \cs $S\subset\bar{C}$ must separate one vertex $v$ from at most two others. Thus $|S|\geq 2$ and we must have $\Nb(v)\subseteq S$. But $\Nb(v)$ induces a $P_{k-2}$ in $C$,
so $\bar{C}[S]$ has at least one non-edge, a contradiction.
\end{proof}
The following is then a direct consequence of the SPGT and Props.~\ref{prop10} and~\ref{prop20}.
\begin{lem}\label{prop:perfect}
  A graph is perfect if and only if all atoms of its \CCD are perfect.\qedhere
\end{lem}
Cobipartite graphs can contain even antiholes. In fact, an even antihole $\bar{C}_{2k}$ is a connected cobipartite graph for any $k\geq3$. 

Thus $\sQ\subseteq\sP$, but $\sQ$ has a stronger inclusion. Recall that $\sP_0$ is the class of long hole-free perfect graphs.
\begin{lem}\label{lem:holesantiholes}
 $\sQ\subset\sP_0$.
\end{lem}
\begin{proof}
  $\sQ\subseteq\sP_0$ follows from Prop.~\ref{prop:ccd} and that cobipartite graphs are in $\sQ\subseteq\sP_0$.
  
  To see that $\sQ\neq\sP_0$, note that otherwise any atom in the \CCD of $G\in\sP_0$ must be cobipartite. Thus we need only exhibit possible atoms which are neither. Complete bipartite graphs $K_{ij}$ with $i\geq 2,j\geq3$ give an example. These are perfect, have no holes and no clique separators, yet are not cobipartite.  
\end{proof}  
   Thus the first of the forbidden subgraphs for thick forests in Fig.~\ref{fig:forb-vc} above is also a forbidden subgraph for $\sQ$, since it is $K_{2,3}$. But none of the other graphs in Figs.~\ref{fig:forb-vc} and~\ref{fig:forb-inf} are forbidden for $\sQ$, since they are chordal and/or have cut edges.
 
\section{Counting independent sets and colourings}\label{sec:counting}
\subsection{Counting independent sets in $\sQ$}\label{sec:indsets}
We consider counting weighted independent sets, meaning evaluating the weighted independence polynomial.
The vertices $v \in \VS(G)$ have non-negative \emph{weights} $w(v)\in\mathbb{Q}$.
Let $\cI_k(G)$ denote the independent sets of size $k$ in $G$, and $\alpha(G)=\max_k \{\cI_k(G)\neq\es\}$.
The weight of a subset $S$ of $V$ is defined to be $w(S)=\prod_{v \in S} w(v)$.
Then let $W_k(G) = \sum_{S \in \cI_k(G)} w(S)$, and $W(G) = \sum_{S \in \cI(G)} w(S)=\sum_{k=0}^n W_k(G)$.
The weighted independence polynomial is $P(\lambda)=\sum_{k=0}^n W_k\lambda^k$, with the $W_k$ as its coefficients,
so $P(\lambda)$ can be regarded as $W(G)$ with weights $\lambda w(v)$.

The following is proved in~\cite[Sec.\,2.2]{DJMV}.
\begin{thm}\label{thm:ccd}
Let $\cC$ be a hereditary class of graphs such that every graph in $\cC$ has a \CCD
with all atoms in the class $\cC_0\subseteq\cC$. Suppose we can evaluate $W(G)$ for any 
$G\in\cC_0$ in time $T_0(n)$, assumed $\Omega(n)$ and convex.  Then we can evaluate $W(G)$ for any $G\in\cC$ in time  $T(n)\leq 2nT_0(n)$.  
\end{thm}
Here \emph{evaluation} means either exactly in polynomial time or approximately using an FPRAS~\cite{Jerrum}. Error control for approximation is described in~\cite[Sec.\,2.2.1]{DJMV}.

The paper~\cite{DJMV} is concerned with evaluating $W(G)$ in claw- and fork-free perfect graphs. Clique cutset decomposition is used for claw-free perfect graphs, where the atoms are either graphs with small independence number or (essentially) line graphs of bipartite graphs, where $W(G)$ can evaluated using the algorithm of~\cite{DJMV}. The extension to fork-free graphs uses a different technique, \emph{modular decomposition}, which we will not discuss here.

Thick forests are perfect, as we have seen, but are not claw-free or fork-free in general. A thick tree $G$ can contain a claw, possibly many, unless the thin tree $H$ is a path, and can contain a fork unless the thin tree $H$ is a star, that is, $G$ is unipolar. Thus thick forests are perfect graphs, but not necessarily claw-free or fork-free, so are not within the class of graphs considered in~\cite{DJMV}. Note that a claw or a fork, in fact any tree, does not fall within the scope of Prop.~\ref{prop10} since they have clique separators of size one.

We can use Theorem~\ref{thm:ccd} to evaluate $W(G)$ for $G\in\sQ$, implying the same result for $\thksf\sF$. The difference from~\cite{DJMV} is in the atoms of the \CCD. Here they are simply cliques and cobipartite graphs. For these we can evaluate the $W_k$ exactly in linear time. For a clique $K$, $W_0=1$, $W_1=\sum_{v\in K}w(v)$, $W_k=0$ ($k>1$). For a cobipartite graph $H$,  $W_0=1$, $W_1=\sum_{v\in H}w(v)$, $W_2=\sum_{uv\notin \ES(H)}w(u)w(v)$ and $W_k=0$ ($k>2$). Thus we can evaluate $W(G)$ exactly for $G\in\sQ$, and there is no need for the error control. Thus we have a deterministic algorithm for counting weighted independent sets in $\sQ$.

A consequence of exact evaluation of $W(G)$ is that we can evaluate $W_k(G)$ exactly for all $0\leq k \leq n$ by interpolation. We can evaluate $P(\lambda)$ by modifying the weights $w(v)$ to $\lambda w(v)$. Since $W_0=1$, if we do this for $\alpha(G)$ positive values of $\lambda$, we can solve $\alpha$ linear equations for the coefficients $W_k$ $(k\in[\alpha])$.

\subsection{Counting colourings in $\sQ$}\label{sec:colouring}
We will consider evaluating $\ncol(G)$ for $G\in\sQ$ using \CCD in a similar way Theorem~\ref{thm:ccd}. 

The proof of Theorem~\ref{thm:ccd} given in~\cite{DJMV} remains valid for any quantity if a clique cutset $G_1\cap G_2$  can be used to lift evaluations of the quantity from $G_1$ and $G_2$ to their union $G=G_1\cup G_2$. For $W(G)$ this uses the simple fact that any clique can only contain one vertex of an independent set. To evaluate $\ncol(G)$ we may replace this by the following \emph{clique cutset colouring lemma}~(CCCL).

\begin{lem}[CCCL]\label{lem:colcut}
Suppose $G=(V,E)$ and $K\subseteq V$ is a clique separator of size $k=|K|$ in $G$.  Let  $G_i=G[V_i]$ ($i\in[2]$),
where $K=V_1\cap V_2$. Then $\ncol(G)=\ncol(G_1)\ncol(G_2)/(q)_k$.
\end{lem}
\begin{proof}
$K$ has $(q)_k$ colourings, each corresponding to permuting a selection of $k$ of the $q$ colours.
Fix a colouring of $K$ and count its extensions to colourings of $G$, $G_1$ and $G_2$. The symmetry of $q$-colourings under permutation of colours implies that these numbers are the same for each colouring of $K$.
However, in the product $\ncol(G_1)\ncol(G_2)$, each colouring of $G$ is counted $(q)_k$ times, so we
correct for this overcounting. We do not need to assume $G$ is connected if we allow $\es$ as a clique separator of size~0.
\end{proof}
Note that the proof of the CCCL relies completely on the symmetries of both $q$-colouring and cliques. Thus it does have any obvious generalisations.

Thus, if we evaluate $\ncol$ for all atoms of the \CCD of $G$, we can use the CCCL to evaluate $\ncol(G)$. This gives the following analogue of Theorem~\ref{thm:ccd}.
\begin{thm}\label{thm:ccdcol}
Let $\cC$ be a hereditary class of graphs such that every graph in $\cC$ has a \CCD
with all atoms in the class $\cC_0\subseteq\cC$. Suppose we can evaluate $\ncol(G)$ for any 
$G\in\cC_0$ in time $T_0(n)$, assumed $\Omega(n)$ and convex.  Then we can evaluate $\ncol(G)$ for any $G\in\cC$ in time  $T(n)\leq 2nT_0(n)$.  
\end{thm}

In~\cite{DJMV} colourings could not be counted since evaluation of \ncol in the atoms that are line graphs of bipartite graphs is not possible in general. That is, we cannot count edge-colourings of bipartite graphs, even approximately.

To apply Theorem~\ref{thm:ccdcol} to $\sQ$, we must consider evaluation of $\ncol(G)$ when $G$ is a clique or cobipartite graph. The number of colourings of a clique of size $k$ is exactly $(q)_k$ from above. Note that if $G$ is a chordal graph, there are no cobipartite atoms, so we have an exact algorithm for counting $q$-colourings of chordal graphs. We could implement this to count colourings in chordal graphs in $\cO(m+n)$ time, but we note this can be done more easily using the algorithm of~\cite[Rem.\,2.5]{Agnar}. 

By contrast, we show below that exact counting is \#P-complete for  cobipartite graphs. So we can only count colourings approximately. If we can approximate $\ncol$ for cobipartite graphs, we can use the CCCL to approximate $\ncol(G)$. We use the following.
\begin{lem}\label{lem:cobipcol}
Let $G$ be a cobipartite graph on $n$ vertices. Let $B=\bar{G}$ and let $\ddot{B}$ be its bipartite complement. Let $\kappa_k(\ddot{B})$ be the number matchings of size $k$ in $\ddot{B}$. Then
\[  \ncol(G)=\sum_{k=n-q}^n \kappa_k(\ddot{B})(q)_{n-k}.\]
\end{lem}
\begin{proof}
Consider $q$-colouring a cobipartite graph $G$ on $n=n_1+n_2$ vertices, having cliques $C_i=(V_i,E_i)$  ($|V_i|=n_i$, $i\in[2]$), connected by a bipartite graph $B=(V_1\cup V_2,E_{12})$.  Then $C_i$ can be coloured in $(q)_{n_i}$ ways, since all vertices must be coloured differently. Also vertices $v_i\in V_i$ ($i\in[2]$) must be coloured differently if $\{v_1,v_2\}\in E_{12}$. Thus vertices in $V_1,V_2$ can only receive the same colour if they form a \emph{matching} in $\ddot{B}$, the bipartite complement of $B$. If $\ddot{B}$ has a matching $M$ of size $k$, $G$ can be coloured in $(q)_{n-k}$ ways, by arbitrarily colouring all vertices not in $M$ and one vertex from each edge in $M$. Thus, if $\ddot{B}$ has $\kappa_k$ matchings of size $k$,  $G$ has $\kappa_k(q)_{n-k}$ colourings corresponding to these. Summing over the possible values of $k$ now gives the result.
\end{proof}
An alternative proof of Lemma~\ref{lem:cobipcol} can be given using the perfection of cobipartite graphs, but it is not simpler than the above.

The following is then an easy consequence of Lemma~\ref{lem:cobipcol}.
\begin{lem}\label{lem:thickcol}
  Counting the number of $q$-colourings of a cobipartite graph is \#P-complete.
\end{lem}
\begin{proof}
 If $n_1=n_2=q$, then $\kappa_k=0$ for $k>q$ and $\ncol(G)= \kappa_k q!$ from Lemma~\ref{lem:cobipcol}. So
 $\ncol(G)/q!$ is the number of perfect matchings of $\ddot{B}$. Thus we have a
 a reduction to $\ncol$ from counting the number of perfect matchings in a bipartite graph.   Valiant~\cite{Valiant} showed this to be \#P-complete.
\end{proof}
Thus we cannot hope for exact counting, but there is an FPRAS (fully polynomial randomised approximation scheme) to approximate the number of $k$-matchings in $G$ to relative error $\eps$ with high probability~\cite{JSV}.  Since Lemma~\ref{lem:cobipcol} gives a linear function of  $\kappa_k$ with positive coefficients, we can use an FPRAS to approximate the $\kappa_k$ to give an FPRAS to approximate $\ncol(G)$ on any cobipartite graph for any $q$. We can then use Theorem~\ref{thm:ccdcol} to give an FPRAS for $\ncol(G)$ for any $G\in\sQ$ and any $q$. 

Finally, note that if the chromatic number $\chi(G)$ for $G\in\sQ$ is a parameter, there is an easy FPT algorithm for counting colourings exactly in $\sQ$. Since $G$ is perfect $\chi$ is the size of the largest clique. Thus the largest clique in $G$ has at most $\chi$ vertices, or else $\ncol(G)=0$. Therefore any atom of the \CCD of $G$ is either a clique of size at most $\chi$ or a cobipartite graph $G_0$ such that $\bar G_0$ has a bipartition with both parts of size at most~$\chi$. Then $\bar G_0$ has at most $\sum_{k=0}^\nu\binom{\nu}{k}^2k!=\cO(\nu^{2\nu})$ matchings of all sizes $k$, which we can count exactly by brute force. We can then use Lemma~\ref{lem:cobipcol} to count the number of colourings of $G_0$. Then we use Theorem~\ref{thm:ccdcol} to extend the counting from the atoms to the whole of $G$. This algorithm has time complexity $\cO(\nu^{2\nu}n)$ from Theorem~\ref{thm:ccdcol}, so is in FPT.

\subsection{Counting colourings of a graph not in $\thksf\sF$}\label{sec:colouringbeyond}
We will show that approximately counting proper colourings of a graph a thin graph which is not a forest is NP-hard. Thus the result of section~\ref{sec:colouring} is best possible.

First we show that
\begin{lem}\label{lem:tricol}
If $\triangle$ is a triangle, it is NP-complete to determine if $T\in\thksf{\triangle}$ has a proper $q$-colouring.
\end{lem}
\begin{proof}
We consider $q$-colouring a $T=(V,E)\in\thksf\triangle$ with three $K_q$ nodes $\bs{a,b,c}$. (See Fig.~\ref{fig:thicktricol}.) Note that the nodes $\bs{a,b,c}$ are fixed and only the links $\bs{ab,bc,ca}$ are the input. Then each of the $q$ colours must occur exactly once in each node. Observe now that $T$ can be properly $q$-coloured if and only if the vertices of $\Tbar=(V,\bar{E})$ can be partitioned into $q$ induced triangles.
But it is shown in~\cite[Prop.\,5.1]{CKW} that it is NP-complete to decide if $\Tbar$ has such a partition,
even if $\Tbar$ is 6-regular or, equivalently, $G$ is ($3q-7$)-regular.
\end{proof}

\begin{figure}[htb]
\centerline{\begin{tikzpicture}[line width=0.33pt]
    \foreach \n\x in {a/90, b/210, c/330}{
    \node at (\x:2.8) {$\bs\n$};
    \foreach \y in {1,2,3,4}{
    \coordinate (\n\y) at (\x:0.5+0.5*\y) {};}}
    \node[v,fill=yellow] (A1) at (a1){}; \node[v,fill=yellow] (B1) at (b1){}; \node[v,fill=yellow] (C1) at (c1) {};
    \node[v,fill=red] (A2) at (a2){}; \node[v,fill=blue] (B2) at (b2){}; \node[v,fill=green] (C2) at (c2) {}; \node[v,fill=green] (A3) at (a3){};\node[v,fill=red] (B3) at (b3){}; \node[v,fill=blue] (C3) at (c3) {};
    \node[v,fill=blue] (A4) at (a4){}; \node[v,fill=green] (B4) at (b4){}; \node[v,fill=red] (C4) at (c4) {};
    \foreach \nn\xx in {A/90,B/210,C/330}{
    \draw (\nn1)--(\nn2)--(\nn3)--(\nn4) ;
    \draw (\nn4) to[out=\xx-125,in=\xx-50] (\nn1);
    \draw (\nn4) to[out=\xx-140,in=\xx-50] (\nn2);
    \draw (\nn1) to[out=\xx+50,in=\xx+120] (\nn3);
    \foreach \xxx in {2,4}{\foreach \yy in {1,4}{
    \draw[line width=0.25pt] (A\xxx)--(B\yy) (B\xxx)--(C\yy) (C\xxx)--(A\yy);}}}
    \end{tikzpicture}}
\caption{A proper 4-colouring of a thick triangle with $K_4$ nodes}
  \label{fig:thicktricol}
\end{figure}

\begin{figure}[hbtp]
\centerline{%
\begin{tikzpicture}[xscale=3,yscale=0.7]
\foreach \j/\k in {1/2,1/3,1/4,2/1,2/3,2/4,3/1,3/2,3/4,4/1,4/2,4/3}{
\draw (1,\j)--(2,\k);};
\foreach \i in {1,2}{
\foreach \j/\k in {4/red,3/blue,2/green,1/yellow}{
\node[fill=\k,v] (\i\j) at (\i,\j) {} ;}
\draw (\i1)--(\i2)--(\i3)--(\i4);};
\draw (12) to[in=-100,out=100] (13)(12) to[in=-105,out=105] (14);
\draw (11) to[in=-100,out=100] (12)(11) to[in=-105,out=105] (13);
\draw (13) to[in=-100,out=100] (14)(11) to[in=-107,out=107] (14);
\draw (23) to[in=80,out=-80] (22)(24) to[in=75,out=-75] (22);
\draw (22) to[in=80,out=-80] (21)(23) to[in=75,out=-75] (21);
\draw (24) to[in=80,out=-80] (23)(24) to[in=73,out=-73] (21);
\end{tikzpicture}}
\caption{A copy link with $q=4$}
  \label{fig:thickcopy}
\end{figure}

We can extend this to arbitrary $H$ using \emph{copy} links. These are cobipartite graphs with both cliques of size $q$, and a bipartite graph which is a $K_{q,q}$ minus a matching. (See~Fig.~\ref{fig:thickcopy}.) The only proper colourings of such a link are those which have the sam colourings of its thick end vertices. 

\begin{lem}\label{lem:Nonforests}
If $H\notin\sF$, it is NP-complete to determine if $G\in\thksf{H}$ has a proper $q$-colouring.
\end{lem}
\begin{proof}
We construct $G\in \thksf H$ as follows. All nodes of $G$ will be $K_q$'s. Let $T\in\thksf\triangle$ be a thick triangle with thick $K_q$ vertices $\bs{a,b,c}$ and input links $\bs{ab}$, $\bs{bc }$, $\bs{ca}$. We will label the vertices of $H$  with \msf{a, b} or \msf c, and associate these labels with the nodes $\bs{a,b,c}$ of $T$. 

If $H\notin\sF$, it has a cycle $C$. Label any three successive vertices of $C$ with \msf a, \msf b and \msf c. Now label all other vertices in $H$ with \msf c. If an edge $uv$ in $H$ is labelled \msf{ab}, let $\thksf{uv}\cong\bs{ab}$ in $G$. If an edge  $uv$ of $H$ is labelled $\msf{xc}$ for $\msf x\in\bsf{a,b}$, we let $\thksf{uv}\cong\bs{xc}$ in $G$. Finally if an edge $uv$ of $H$ is labelled \msf{cc}, let $\thksf{uv}$ be a copy edge in $G$. Now it is clear that $G$ has a proper $q$-colouring if and only if $T$ has a proper $q$-colouring.
\end{proof}
Thus we have the following. Note that the class $\sC$ is not restricted to be hereditary.
\begin{thm}\label{thm:nocol}
  If $H\notin\sF$, then counting proper $q$-colourings of graphs in $\thksf H$  is complete for {\rm\#P} under {\rm AP}-reducibility.
\end{thm}
\begin{proof}
  Follows directly from Lemma~\ref{lem:Nonforests} and \cite[Thm.\,1]{DyGoGJ}.
\end{proof}

\section{Parameterisation}\label{sec:parameter}
Thick forests are a tractable class, but this does not seem true for most classes, in the light of Lemma~\ref{lem:thickNP}. However, we can consider fixed-parameter tractability, as we did for exactly counting colourings in $\sQ$ in section~\ref{sec:colouring}. Here we consider two possible parameterisations by properties of the thin graph. 
\subsection{Parameterising by the size of the thin graph}\label{sec:fixedparam}
The most natural parameter for thick graphs is the size of the thin graph. Even so, we must restrict attention to triangle-free graphs, the class~$\sT$. Otherwise, the recognition problem is NP-complete~\cite{MacYu}. Let us denote the class of triangle-free graphs with at most $\nu$ vertices by $\sT_\nu=\sT\cap\sG_\nu$. Then, given a graph $G$,  we consider deciding whether $G\in\thksf{\sT_\nu}$ with $\nu$ as a parameter.

The case where $\sC=\{H\}$ for a fixed graph $H$ was resolved by MacGillivray and Yu~\cite{MacYu}. They proved the following, which we recast in our own notation.
\begin{thm}[MacGillivray-Yu]\label{thm:MacYu}
Let $\nu=|\VS(H)|$. If $H\in\sT$, the decision problem $G\in\thksf H$ with parameter $\nu$ is in XP. Otherwise it is NP-complete.\qed
\end{thm}
This shows NP-completeness of recognising $H\in\thin G$ for 
$G\notin\thksf\sT$, but shows that there is a recognition algorithm if the 
parameter is $H\in\sT$. The algorithm given in~\cite{MacYu} has complexity 
$n^{\Omega(\nu)}$, so is only an XP algorithm. An obvious question is whether 
this can be improved to an FPT algorithm. We note that this has been done for 
thick bipartite graphs in~\cite{KKSL}, where an FPT algorithm is given for 
the stronger parameterisation that $\nu$ is the size of the smaller part of 
the bipartition. This gives a different generalisation of unipolar graphs 
from \thksf\sF. 

Theorem~\ref{thm:MacYu} answers a slightly different question from deciding whether $G\in\thksf{\sT_\nu}$, so we first show equivalence for fixed-parameter tractability.
\begin{lem}\label{lem:equivalent}
  Recognition of $G\in\thksf H$ with parameter $H$ is in FPT if and only if recognising $G\in\thksf{\sT_\nu}$ with parameter $\nu$ is in FPT.
\end{lem}
\begin{proof}
  Suppose we have a FPT algorithm with parameter $H\in\sT$ to decide whether $G\in\thksf H$, where $\nu=|\VS(H)|$. Thus we can test whether $G\in\thksf{H'}$ for each of the graphs $H'\in\sT_\nu$  using the given FPT algorithm. We simply generate all $2^{\nu(\nu-1)/2}$ graphs on $\nu$ vertices and reject those with triangles. This gives an FPT algorithm with parameter~$\nu$ to decide if $G\in\thksf{\sT_\nu}$. 
  
  Conversely, suppose we have an FPT algorithm with $\nu$ as parameter which 
  returns $H'\in\thin G\cap\sT_{\nu}$. Since $G\in\thksf\sT$, its neighbourhoods are unipolar. 
  Then vertices can only be moved into one of the identified neighbours of any node $\bs u$,
  from Cor.~\ref{cor:neighbour}. It follows that $H'$ is the same as $H$ up to contracting full links 
  and/or splitting nodes. So any other graph $H''\in \thin G\cap\sT_\nu$ can be formed by 
  contracting the links and/or splitting the nodes of $G$. We generate all such graphs 
  $H''$ by performing sequences of valid splits and contractions. 
  
  There are at most $2^{|\ES(H')|}\leq2^{\nu^2}$ ways of contracting full links of $G$. 
  Since $\dg(\bs u)< \nu$ for $\bs u\in H'$, there are less than 
  $2^{\nu}$ ways of splitting $\bs u$ so that the neighbourhoods of its two 
  parts are disjoint, and thus $2^{\nu^2}$ overall. Thus we generate 
  less than $2^{2\nu^2}$ graphs $H''$. If $H''\notin\sT$ or $|\VS(H'')|\neq\nu$ 
  then $H''$ can be discarded. For the remainder, we check if $H''$ 
  is isomorphic to $H$. At worst, this requires comparing all $\nu!$ 
  relabellings of $H$ with $H''$. This gives an FPT algorithm for recognising 
  $G\in\thksf H$. 
\end{proof}
We now show that
\begin{thm}\label{thm:fpt}
  Recognition of $G\in\thksf{\sT_\nu}$ with parameter $\nu$ is in FPT.
\end{thm}
\begin{proof}
First we check whether all components of $G$ are all cliques and cobipartite graphs,
using the method of section~\ref{sec:cobip}. If so, the thin graph $H$ has only isolated verticess and edge.
If so, we accept \ifff $H\in\sG_\nu$.

Otherwise, we choose $u\in V$ and use the method for recognising a node $\bs u$ of $G\in\thksf\sF$ given in section~\ref{sec:recog}, which applies equally to $G\in\thksf\sT$.  If this fails, we reject $G$ and stop.

Now $\Nb(\bs u)$ may have vertices misplaced to one of its thick neighbours $\bs w_i$ for $i\in[d]$, where $d=\deg(\bs u)$. In recognising $G\in\thksf\sF$ this misplacement is not problematic, since its effect can only propagate along a path in $H$. Then we can recover from the misplacement as described in section~\ref{sec:adjust}. But if $H\notin\sF$, these paths of misplacements may collide, and efficient recovery may not be possible. However, suppose we have misplacement into one of $\bs u$'s thick neighbours, $\bs w_j$ $(j\in[d])$. 

Now $\Nb(\bs u)$ gives the edge sets $\bs u:\bs w_i$ for $i\in[d]\sm j$ so deleting these from $G$ gives a graph $G'_j$  with $(d-1)$ fewer links and one more node with degree at most 1. Now $G_j'\in \thksf{\sT_\nu}$ if and only if $G\in \thksf{\sT_\nu}$. We now repeat the above on $G'_j$. This cannot continue for $\nu$ iterations or all vertices in $\thin{G''}$ for the resulting graph $G''$ have degree at most~1. Then $\thin{G''}$ comprises only isolated links and vertices, and will be accepted at the start of the next iteration.

The problem is that we may not know $j$, so we must try all $d$ possibilities, giving rise to $d$ graphs $G'_1$, \ldots, $G'_d$. Now $G\in\thksf\sT$ if and only if at least one of these is in $\thksf\sT$. Since $d<\nu$ and at most $\nu$ iterations are necessary, at most $\nu^\nu$ graphs are generated before acceptance or rejection. Since the required computation on each such graph runs in time in $\cO(mn)$, from section~\ref{sec:recog}, the algorithm is in FPT with parameter $\nu$.
\end{proof}
\begin{cor}\label{cor:subdivision}
 The graph operation of replacing single edges by paths is called \emph{subdivision}. Since we can fix a path, using the method of section~\ref{sec:adjust}, it follows that the algorithm of Theorem~\ref{thm:fpt} gives polynomial time recognition of the class of graphs which are thick subdivisions of triangle-free graphs on $\nu$ vertices. In particular, this includes all thick cycles of length at least 4 since these are subdivisions of a thick cycle of length 4, which is in $\thksf{\sT_4}$.\qed
\end{cor}
\begin{rem}\label{rem:smallbip}
In~\cite[Thm.\,1.4]{KKSL}, Theorem~\ref{thm:fpt} is proved for the special case $G\in\thksf{K_{\nu_1,\nu_2}}$ and parameter $\nu=\nu_1+\nu_2$. In fact, \cite{KKSL}  give a stronger result~\cite[Thm.\,1.3]{KKSL}: there is an FPT algorithm when the parameter is $\nu_1$ only. This clearly does not follow immediately from Theorem~\ref{thm:fpt}. We believe our methods can be adapted to give this improvement, but we will not pursue it further here.
\end{rem}
The criterion of Theorem~\ref{thm:MacYu} extends to any finite class $\sC$, but not to infinite classes, though we know from Lemma~\ref{lem:trianglerecog} that $\thksf\sC$ is not recognisable for any hereditary class $\sC$ such that $\sC\not\subseteq\sT$. We have seen that recognition of thick forests $\thksf\sF\subset\thksf\sT$ is in P, but for thick bipartite graphs $\thksf\sB\subset\thksf\sT$ recognition is NP-complete. Neither is being cycle-free a necessary condition, from Cor.~\ref{cor:subdivision}. We leave as an open question what properties of $\sC$ are necessary and sufficient for $G\in\thksf\sC$ to be decidable in P.

As for algorithms on $\sT_\nu$, counting independent sets in the class of Theorem~\ref{thm:fpt} is in XP, since any of the $\nu$ nodes may contain at most one vertex of the independent set. Thus we can simply check all of the $\cO(n^\nu)$ possibilities. Though we will not prove it here, this problem is $\textrm{\#W}[1]$-complete. (See, for example,~\cite[Prop.\,4.2]{KKSL} for a similar result.) Thus we cannot hope for an FPT algorithm.

Counting colourings, even approximately, in $G\in\thksf\sT_\nu$ is clearly not tractable unless $G\in\thksf\sF$, by Theorem~\ref{thm:nocol}.

\subsection{Parameterising by treewidth}\label{sec:treewidth}
Here we consider parameterising by the \emph{treewidth} of the thin graph.
Treewidth measures how treelike a graph is. Let $\sW_k$ be the class of graphs with
treewidth at most $k$. 

In section~\ref{ss:btw def} we will review the ideas. 
We note that computing treewidth is NP-hard, even when restricted to
co-bipartite graphs~\cite{ACP}. There is an $\cO(n^{k+2})$ XP algorithm~\cite{ACP} 
for obtaining a tree decomposition of width at most $k$, if one exists, 
but known FPT algorithms only give approximations~\cite{BelFur}.

Treewidth is an important parameterisation of graphs, because of its
application in algorithmic meta-theorems~\cite{MSOLiii}.
Our interest in $\sW_k$ here is that the class $\sW_1=\sF$, and we have shown that $\thksf\sF$ is a tractable class. So we ask which of our results above can be extended to $\thksf{\sW_k}$ for small $k>1$.

Unfortunately, the classes \thksf{\sW_k} already have an NP-complete recognition problem for $k>1$ by Lemma~\ref{lem:trianglerecog}, since $\sW_k$ includes the triangle for all $k>1$. Whether the class $\thksf{\sW_2\cap\sT}$ has polynomial time recognition we leave as an open question, but this seems possible, since $\sW_2$ contains only series-parallel graphs. In particular, $\sW_2\cap\sT$ contains all cycles $C_\ell$ with
$\ell>3$, which do have polynomial time recognition from~Cor.~\ref{cor:subdivision}.

Though we cannot recognise $G\in\thksf\sW_k$ in general, we may still be able to use this characterisation of $G$ for algorithmic purposes provided a decomposition of $G$ with thin graph $H\in\sW_k$ is given. This is true of many other NP-complete graph parameters, clique-width~for example~\cite{KLM}. In this vein, we will show in section~\ref{ss:btw is} that we can count independent sets exactly in $\thksf{\sW_k}$.

As for counting colourings, we know by Lemma~\ref{lem:Nonforests} that this is intractable for $G\notin\thksf\sF$, that is, $G\notin\thksf{\sW_1}$. So counting colourings in \thksf{W_k} is intractable for any $k>1$.

\subsubsection{Counting independent sets}\label{sec:indsetsbeyond}\label{ss:btw is}

Here we show that there is a deterministic algorithm for computing $\nis(G)$, the number of
independent sets of a thick graph $G$ if we are given a model $(H,\psi)$
with $\tw(H)\leq k$ for some~$k>1$.

If $\alpha(G)$ is the size of the largest independent set in $G$,
the algorithm is based on the following lemma.

\begin{lem}\label{l:base}
  Let $G=(V,E)$ be a graph with model $(H,\psi)$ and let $(I,b)$ be a
  tree decomposition of $H$ of width $k$. For every $i \in I$ we have
  $\alpha(G[\psi^{-1}(b(i))]) \le k+1$.
\end{lem}

\begin{proof}
  The bag $b(i)$ contains at most $k+1$ vertices of $H$ since the width of
  $(I,b)$ is $k$. Each vertex $\bs{v} \in b(i)$ corresponds to a clique
  $\psi^{-1}(\bs{v})$ of $G$ by property~\eqref{m:ver} of the model $(H,\psi)$.
\end{proof}
We can now show
\begin{lem}\label{l:base}
  Let $G=(V,E)$ be a graph with model $(H,\psi)$, where $H\in\sW_k$. Then there is an $\cO(n^{k+2})$
  algorithm for counting all independent sets in $G$.
\end{lem}
\begin{proof}
In $\cO(n^{k+2})$ time we can find a nice tree decomposition $(T,b)$ of $H$ of width $k$,
using the algorithm of~\cite{ACP}. Let $T=(I,A)$ with root $r$. For each node $i \in I$ let $C(i)$ be
the set of descendants of $i$ in $T$. These are the nodes $j \in I$ such that
$i$ is on the path from $j$ to $r$ in $T$, including $i$ itself. Our algorithm
computes recursively, for every $i \in I$ and every independent set
$S \subseteq \psi^{-1}(b(i))$ of $G$, the numbers $a(i,S)$ of
independent sets of $G(i) = G[\psi^{-1}\big(\bigcup_{j \in C(i)} b(j)]\big)$
that coincide with $S$ on $\psi^{-1}(b(i))$. That is,
$\nis(G(i)) = \sum_{S} a(i,S)$, where we sum over all independent sets
$S \subseteq \psi^{-1}(b(i))$ of $G(i)$. The values of $a(i,S)$ are computed
from the leaves of $T$ up to the root $r$ as follows.

\begin{description}[topsep=0pt,itemsep=0pt]
\item[leaf] For every leaf $l$ of $T$ we have $a(l,\es)=1$, because the empty
  set $\es$ is the only independent set in the empty graph $(\es,\es)$.
\item[introduce] Let $i$ be an introduce node with child $j$ and let
  $\bs{w} \in b(i) \sm b(j)$. For every independent set
  $S\subseteq \psi^{-1}(b(j))$ and every vertex $v \in \psi^{-1}(\bs{w})$
  such that $S\cup\{v\}$ is independent we have $a(i,S) = a(j,S)$ and
  $a(i,S\cup\{v\}) = a(j,S)$. The former is obvious. The later equality follows
  from property \ref{tw:edg} of tree decompositions: Since $\bs{w}$ is the new
  vertex in $b(i)$, all neighbours of $v$ in $G(i)$ are in $\psi^{-1}(b(i))$.
\item[forget] Let $i$ be an forget node with child $j$ and let
  $\bs{w} \in b(j) \sm b(i)$. For every independent set
  $S\subseteq \psi^{-1}(b(j))$ and every $v \in \psi^{-1}(\bs{w})
  $ we have $a(i,S) = a(j,S)$ if $v \notin S$ and
  $a(i,S\sm\{v\}) = a(j,S\sm\{v\}) + a (j,S)$ if $v \in S$.
\item[join] Let $i$ be a join node with children $j$ and $j'$. For every
  independent set $S \subseteq \psi^{-1}(b(i))$ we have
  $a(i,S) = a(j,S) \cdot a(j',S)$, because there are no edges of $H$ between
  vertices in $C(j)\sm b(i)$ and $C(j')\sm b(i)$, and consequently there are
  no edges on $G$ between $\psi^{-1}(C(j)\sm b(i))$ and
  $\psi^{-1}(C(j')\sm b(i))$.
\item[root] $G$ has $\nis(G) = a(r,\es)$ independent sets because
  $S \cap b(r) = \es$ holds for all independent sets $S$ of $G$.
\end{description}
For an $n$-vertex  graph $G$ we must compute a table of $\cO(n^{k+2})$ values
$a(i,S)$ because the tree $T$ has linear size, and by Lemma \ref{l:base}
the independent set $S$ chooses at most $k+1$ vertices from the set
$\psi^{-1}(b(i))$. Consequently, the algorithm runs in time $\cO(n^{k+2})$. 
\end{proof}
The algorithm can be modified easily to deal with weighted independent sets.

Thus we have an XP algorithm provided we are given an $H\in\thin G$ of treewidth at most $k$. 
We leave as an open question whether there is an FPT algorithm for counting independent sets in this setting.

\section*{Acknowledgment}
We thank Mark Jerrum for useful inputs and comments on an earlier version of this paper.


\begin{thebibliography}{99}

\bibitem{Agnar} G. Agnarsson (2003), On chordal graphs and their chromatic polynomials, \emph{Mathematica Scandinavica} \textbf{93}, 240--246.
    
\bibitem{ACP} S.~Arnborg, D.~Corneil and A.~Proskurowski (1987),
  Complexity of finding embeddings in a $k$-tree,
  \emph{SIAM Journal on Algebraic Discrete Methods} \textbf{8}, 277--284.

\bibitem{AJHL} M.~Albertson, R.~Jamison, S.~Hedetniemi and S.~Locke (1989),
The subchromatic number of a graph, \emph{Discrete Mathematics} \textbf{74}, 33--49.

\bibitem{BPS} A. Berry, R. Pogorelcnik, and G. Simonet (2010), An introduction to clique minimal separator decomposition, \emph{Algorithms} \textbf{3}, 197--215.

\bibitem{BPV} V. Boncompagni, I. Penev and K. Vu\v{s}kovi\'c (2019), Clique-cutsets beyond chordal graphs, \emph{Journal of Graph Theory} \textbf{91}, 192-246.
    
\bibitem{BovCre} D. Bovet and P. Crescenzi (2006), Introduction to the theory of complexity,
\url{https://www.pilucrescenzi.it/wp/wp-content/uploads/2017/09/itc.pdf}.

\bibitem{BelFur} M. Belbasi and M. Fürer (2022), An improvement of Reed's treewidth approximation, 
\arxiv{2010.03105}.
  
\bibitem{BFNG}  H. Broersma, F. Fomin, J. Ne\v set\v ril and G. Woeginger (2002), More about subcolorings, \emph{Computing} \textbf{69}, 187–203.

\bibitem{Brown} J.~Brown (1996), The complexity of generalized graph colorings, \emph{Discrete Applied Mathematics} \textbf{69}, 257-270.

\bibitem{ChRoST} M. Chudnovsky, N. Robertson, P. Seymour and R. Thomas (2006), The strong perfect graph theorem,
  \emph{Annals of Mathematics} \textbf{164}, 51--229.

\bibitem{CSSS} M. Chudnovsky, A. Scott, P. Seymour and S. Spirkl (2020), Detecting an odd hole, \emph{Journal of the ACM} \textbf{67}, Article 5, 1--12.

\bibitem{CKW} A. \'{C}usti\'{c}, B. Klinz and G. Woeginger (2015),
Geometric versions of the three-dimensional assignment problem under general norms,
\emph{Discrete Optimization} \textbf{18}, 38--55.

\bibitem{CohJea} D. Cohen and P. Jeavons (2006), The complexity of constraint languages, in \emph{Handbook of Constraint Programming}, Elsevier, pp.~169--204.
    
\bibitem{CLV} G. Cornuejols, X.~Liu and K. Vu\v{s}kovi\'{c} (2003), A polynomial algorithm for recognizing perfect graphs, in \emph{Proc. 44th IEEE Symposium on Foundations of Computer Science (FOCS2003)},  pp. 20--27.
    
\bibitem{MSOLiii} B.~Courcelle (1992),
  The monadic second-order logic of graphs III:
  Tree-decompositions, minors and complexity issues,
  \emph{RAIRO.\ Informatique Th\'eorique et Applications} \textbf{26}, 257--286.

\bibitem{diestel} R.~Diestel, \emph{Graph Theory}, 4th edition, Graduate Texts in Mathematics \textbf{173}, Springer, 2012.

\bibitem{DyGoGJ} M. Dyer, L.A. Goldberg, C. Greenhill and M. Jerrum (2004),
The relative complexity of approximate counting problems, \emph{Algorithmica} \textbf{38}, 471--500.

\bibitem{DyeGre} M.~Dyer and C.~Greenhill, The complexity of counting graph homomorphisms, \emph{Random Structures \& Algorithms} \textbf{17} (2000), 260--289.
    
\bibitem{DGM}  M. Dyer, C. Greenhill and H. M\"uller (2021), Counting independent sets in graphs with bounded bipartite pathwidth, \emph{Random Structures \& Algorithms} \textbf{59}, 204--237.

\bibitem{DJMV} M. Dyer, M. Jerrum,  H. M\"{u}ller  and K. Vu\v{s}kovi\'{c} (2021), Counting weighted independent sets beyond the permanent, \emph{SIAM Journal on Discrete Mathematics} \textbf{35}, 1503--1524.

\bibitem{EW} E. Eschen and X. Wang (2014), Algorithms for unipolar and generalized split graphs, \emph{Discrete Applied Mathematics} \textbf{162}, 195--201.

\bibitem{EHPS} E. Eschen, C. Ho\`{a}ng, M. Petrick and R. Sritharan (2005), Disjoint clique cutsets in graphs without long holes, \emph{Journal of Graph Theory} \textbf{48}, 277--298.
    
\bibitem{Fellows} M. Fellows (2002), Parameterized complexity: the main ideas and connections to practical computing, \emph{Electronic Notes in Theoretical Computer Science} \textbf{61}, 1--19.
    
\bibitem{FJLS} J.  Fiala,  K.  Jansen,  V.B.  Le  and  E.  Seidel (2003),   Graph  subcolorings:  complexity  and algorithms, \emph{SIAM Journal on Discrete Mathematics} \textbf{16}, 635–-650.

\bibitem{farrugia} A.~Farrugia (2004), Vertex-partitioning into fixed additive induced-hereditary properties is NP-hard, \emph{Electronic J. Combinatorics} \textbf{11}, Research Paper \#46.
    
\bibitem{FederHell} T. Feder and P. Hell (2006), Matrix partitions of perfect graphs, \emph{Discrete Mathematics} \textbf{306}, 2450--2460.
    
\bibitem{THJS}    T. Feder, P. Hell, J. Stacho, and G. Schell (2011), Dichotomy for tree-structured matrix partition problems, \emph{Discrete Applied Mathematics}  \textbf{159} 1217--1224.

\bibitem{GLS} M. Gr\"otschel, L. Lov\'asz and A. Schrijver, \emph{Geometric algorithms and combinatorial optimization}, Springer-Verlag, 1988.

\bibitem{Gavril} F. Gavril (1974), The intersection graphs of subtrees in trees are exactly the chordal graphs, \emph{Journal of Combinatorial Theory Series B} \textbf{16}, 47--56.
    
\bibitem{Hayward} R. Hayward (1985), Weakly triangulated graphs, \emph{Journal of Combinatorial Theory B}  \textbf{39}, 200--209.
    
\bibitem{Hoang} C.~Hoang (2010),
On the complexity of finding a sun in a graph,
\emph{SIAM Journal on Discrete Mathematics} \textbf{23}, 2156--2162.
    
\bibitem{Jerrum}
M.~Jerrum, \emph{Counting, sampling and integrating: algorithms and complexity},
Lectures in Mathematics -- ETH Z\"{u}rich, Birkh\"{a}user, 2003. 

\bibitem{JSV} M. Jerrum, A. Sinclair and E. Vigoda (2004), A polynomial-time approximation algorithm for the permanent of a matrix with non-negative entries, \emph{Journal of the ACM} \textbf{51}, 671--697.
    
\bibitem{KLM} M. Kami\'{n}ski, V. Lozin and M. Milani\v{c} (2009),
Recent developments on graphs of bounded clique-width,
\emph{Discrete Applied Mathematics} \textbf{157}, 2747--2761.
  
\bibitem{KKSL} I. Kanj, C. Komusiewicz, M. Sorge and E. van Leeuwen (2018),
Parameterized algorithms for recognizing monopolar and 2-subcolorable graphs,
\emph{Journal of Computer and System Sciences} \textbf{92}, 22--47.

\bibitem{Kloks} T.~Kloks, Treewidth -- computations and approximations.
  Springer-Verlag Berlin, \textsl{Lecture Notes in Computer Science}
  vol.~\textbf{842}, 1994.

\bibitem{Lovasz} L. Lov\'{a}sz (1972), Normal hypergraphs and the perfect graph conjecture, \emph{Discrete Mathematics} \textbf{2}, 253--267.

\bibitem{MacYu} G. MacGillivray and M.L. Yu (1999), Generalized partitions of graphs,
\emph{Discrete Applied Mathematics} \textbf{91}, 143--153.

\bibitem{McYo} C. McDiarmid and N. Yolov (2016), Recognition of unipolar and generalised split graphs, \emph{Algorithms} \textbf{8}, 46--59.

\bibitem{NikPal} S. Nikolopoulos and L. Palios (2007), Hole and antihole detection in graphs, \emph{Algorithmica} \textbf{47}, 119--138.

\bibitem{RTL} D. Rose, R. Tarjan and G. Lueker (1976), Algorithmic aspects of vertex elimination on graphs,
\emph{SIAM Journal on Computing} \textbf{5}, 266--283.

\bibitem{SalJ} A. Salamon and P. Jeavons (2008), Perfect constraints are tractable, in \emph{Proc. 14th International Conference on Principles and Practice of Constraint Programming (CP 2008)}, LNCS \textbf{5202}, pp.~524--528.
    
\bibitem{Stacho} J. Stacho, \emph{Complexity of generalized colourings of chordal graphs}, Chapter 5, Ph.D. thesis, Simon Fraser University, 2008. See also J. Stacho (2008), On 2-subcolourings of chordal graphs, in: \emph{LATIN 2008 : Theoretical Informatics}, Springer LNCS~\textbf{4957}, pp.~544--554.

\bibitem{Tarjan} R. Tarjan (1985), Decomposition by clique separators, \emph{Discrete Mathematics} \textbf{55}, 221--232.
    
\bibitem{Tysh} R. Tyshkevich and A. Chernyak (1985), Algorithms for the canonical decomposition of a graph and recognizing polarity, Izvestia Akad. Nauk BSSR, ser. Fiz. Mat. Nauk 6, 16--23. (in Russian)

\bibitem{Valiant} L. Valiant (1979), The complexity of computing the permanent, \emph{Theoretical Computer Science}~\textbf{8}, 189--201.
   
\bibitem{Yannakakis} M. Yannakakis (1981), Computing the minimum fill-in is NP-complete, \emph{SIAM J. Alg. Disc. Meth.}~\textbf{2}, 77--79.
\end{thebibliography}
\end{document}

\begin{figure}[htb]
  \centering
  \begin{tikzpicture}[xscale=0.75,yscale=0.75,font=\sffamily,font={\sffamily\scriptsize}]
    \node[v] (a) at (0,0) {a};
    \node[v] (b) at (2,0) {b};
    \node[v] (c) at (1,1) {c};
    \node[v] (d) at (0,2) {d};
    \node[v] (e) at (2,2) {e};
    \draw (a)--(d)--(e)--(b)--(a);
    \draw (a)--(c)--(d) (b)--(c)--(e);
  \end{tikzpicture}\caption{Nonchordal $G_C$}\label{fig:nonchordal} 
\end{figure}